\documentclass[11pt,reqno]{amsart}
\usepackage{amsfonts,amssymb}
\usepackage[alphabetic]{amsrefs}
\usepackage{tikz-cd,microtype,lmodern,hyperref,tensor}
\usepackage[headings]{fullpage}

\newcommand{\N}{\mathbb{N}}

\newcommand{\A}{\mathcal{A}}

\renewcommand{\phi}{\varphi}

\theoremstyle{plain}
\newtheorem{theorem}[subsection]{Theorem}
\newtheorem{proposition}[subsection]{Proposition}
\newtheorem{lemma}[subsection]{Lemma}
\newtheorem{corollary}[subsection]{Corollary}

\newtheorem*{theorem*}{Theorem}

\theoremstyle{definition}
\newtheorem{definition}[subsection]{Definition}

\newtheorem{remark}[subsection]{Remark}

\newtheorem{conjecture}[subsection]{Conjecture}

\numberwithin{equation}{section}

\begin{document}

\title{Quasi-invariant lifts of completely positive maps for groupoid actions}
\author{Suvrajit Bhattacharjee} 
\address{Matematisk institutt, Universitetet i Oslo, P.O. Box 1053, Blindern, 0316 Oslo, Norway}
\email{suvrajib@math.uio.no}

\author{Marzieh Forough} 
\address{Department of Applied Mathematics, Faculty of Information Technology, Czech Technical University in Prague, Th\'akurova 9, 160 00, Prague 6, Czech Republic} 
\email{foroumar@fit.cvut.cz}

\begin{abstract}
Let $G$ be a locally compact, Hausdorff, second countable groupoid and $A$ be a separable, $C_0(G^{(0)})$-nuclear, $G$-$C^*$-algebra. We prove the existence of quasi-invariant, completely positive and contractive lifts for equivariant, completely positive and contractive  maps from $A$ into a separable, quotient $C^*$-algebra. Along the way, we construct the Busby invariant for $G$-actions.
\end{abstract}



\maketitle

\section{Introduction} In a series of papers \cites{szabo:self-absorbing-1,szabo:self-absorbing-2,szabo:self-absorbing-3}, Szab\'o introduced and in the process, systematized, strongly self-absorption of dynamical systems for actions of locally compact, Hausdorff groups. Motivated by the work of \cite{hirshberg-rordam-winter:self-absorption} on strongly self-absorption of upper semi-continuous fields of $C^*$-algebras, the second author and Gardella were naturally led in \cite{forough-gardella:absorption} to extend Szab\'o's framework to Rieffel's upper semi-continuous fields of actions, see \cite{rieffel:continuous-fields}*{Definition 3.1}. The second author and Gardella had the need for and proved, together with Thomsen, the following (special case of the) lifting theorem in \cite{forough-gardella-thomsen:lifting}, which is also the main technical tool in \cite{forough-gardella:absorption}. 

\begin{theorem}[{\cite{forough-gardella-thomsen:lifting}*{Theorem 3.4}}]\label{theo:equivariant-lifts}
 Let $G$ be a locally compact, Hausdorff, second countable group and $(A,\alpha)$ be a separable, nuclear, $G$-$C^*$-algebra. Let  
\[
0 \rightarrow (I,\gamma) \rightarrow (B,\beta) \xrightarrow{p} (D,\delta) \rightarrow 0
\]be a $G$-$C^*$-extension, with $B$ separable. Let $\psi : A \rightarrow D$ be a $G$-equivariant, completely positive and contractive map. Then there exists a sequence of completely positive and contractive maps $\phi_n : A \rightarrow B$, $n \in \mathbb{N}$, such that $p \circ \phi_n=\psi$ and 
\begin{equation*}
\lim_{n \rightarrow \infty} \sup_{g \in K} \left\lVert \beta_g (\phi_n(a))-\phi_n(\alpha_g(a)) \right\rVert=0,
\end{equation*}for all $a \in A$ and for all compact subsets $K \subseteq G$.
\end{theorem}

It is well-known by now that Arveson, in his work \cite{arveson:essentially-normal}, used previous lifting theorems of Vesterstr\o m, Anderson and Ando, to prove that the Brown-Douglas-Fillmore group $\mathrm{Ext}(X)$ for a compact metric space $X$ (\cite{brown-douglas-fillmore:k-homology}), is indeed a group. Soon afterwards, Choi-Effros in their remarkable paper \cite{choi-effros:lifting} generalized all the previously known lifting theorems for completely positive (and contractive) maps. Within a year, Arveson, in his seminal paper \cite{arveson:extensions}, introduced the ubiquitous quasi-central approximate units and used it firstly, to simplify Choi and Effros' proof of their lifting theorem and secondly, to show that the Brown-Douglas-Fillmore group $\mathrm{Ext}(A)$ for a separable, nuclear $C^*$-algebra $A$ is again indeed a group, which was previously shown to be a semigroup by Voiculescu. Using similar ideas, Kasparov in \cite{kasparov:operator-k} showed that $\mathrm{Ext}(A,B)$ for a separable, nuclear $C^*$-algebra $A$ and a stable, $\sigma$-unital $C^*$-algebra $B$ is a group too. By now, the Choi-Effros lifting theorem is a standard, important technical tool, which also plays a crucial role in the recent breakthrough \cite{tikuisis-white-winter:quasidiagonality}. It may also be perceived as a noncommutative analogue of the Michael's selection theorem \cite{michael:selection}; see, for instance \cites{kirchberg:classification,gabe:lifting}.   

Although, the framework in \cite{forough-gardella:absorption} is general enough to include most of the cases of interest, it is, however, in many ways natural to let the group act on the base of the field of $C^*$-algebras, i.e., to consider $G$-$C_0(X)$-algebras in Kasparov's sense, \cite{kasparov:equivariant-kk}*{Definition 1.5}. As a first step towards this direction, one would need to establish Theorem \ref{theo:equivariant-lifts} for such algebras which does not follow automatically. The difficulty maybe perceived from the observation that Theorem \ref{theo:equivariant-lifts} enables one to obtain only fiber-wise, quasi-invariant lifts; one then has to glue them in order to construct a global lift which is nearly impossible except for a few pathological cases. Another scenario, which is again natural, is to consider, roughly speaking, dynamical systems $(A(x),G(x),\alpha(x))$ varying continuously on some parameter space $X$. The obstruction is similar; one only has fiber-wise, quasi-invariant lifts.

Our approach is to take the more general vantage point of groupoids and along the way, include foliations into the framework. Introduced in operator algebras more than four decades ago by Connes, by now, there is a substantial literature on groupoids and their dynamical systems. The dynamical flavour of groupoids was investigated in \cite{renault:crossed-products} probably for the first time and studied from the perspective of Kasparov's bivariant theory in \cite{le_gall:equivariant-kk}. The latter was then successfully employed in \cite{tu:baum-connes-amenable} in the context of the Baum-Connes conjecture. We take the very first steps in this article of a systematic exploration of strongly self-absorption for groupoid dynamical systems, aiming towards generalizations of \cites{szabo:equivariant-kirchberg-phillips,suzuki:equivariant-absorption} and the recent breakthrough papers \cites{gabe-szabo:dynamical-kirchberg-phillips,gabe-szabo:stable-uniqueness}. Let us now state (a special case of) our main theorem, which is a vast generalization of Theorem \ref{theo:equivariant-lifts}.  

\begin{theorem}\label{theo:main-result}
Let $G$ be a locally compact, Hausdorff, second countable groupoid. Let $(A,\alpha)$ be a separable, $C_0(G^{(0)})$-nuclear, $G$-$C^*$-algebra and let  
\[
0 \rightarrow (I,\gamma) \rightarrow (B,\beta) \xrightarrow{p} (D,\delta) \rightarrow 0
\]be a $G$-$C^*$-extension with $B$ separable. Let $\psi : A \rightarrow D$ be a $G$-equivariant, $C_0(G^{(0)})$-linear, completely positive and contractive map. Then there exists a sequence of $C_0(G^{(0)})$-linear, completely positive and contractive maps $\phi_n : A \rightarrow B$, $n \in \mathbb{N}$, such that $p \circ \phi_n=\psi$ and 
\begin{equation*}
\lim_{n \rightarrow \infty} \sup_{g \in K} \left\lVert \beta_g ((\phi_n)_{s(g)}(a(s(g))-(\phi_n)_{r(g)}(\alpha_g(a(s(g)))) \right\rVert=0,
\end{equation*}for all $a \in A$ and for all compact subsets $K \subseteq G$.
\end{theorem}

We refer the reader to the next section for the undefined terminology and end this introduction with a few words on the proof. Our approach is standard and is already taken in \cite{forough-gardella-thomsen:lifting} to prove Theorem \ref{theo:equivariant-lifts}. Roughly speaking, $C_0(G^{(0)})$-nuclearity ensures the existence of a $C_0(G^{(0)})$-linear, completely positive and contractive lift. We then ``average'' this lift with respect to a quasi-invariant, quasi-central, approximate units. This is achieved by first considering the case when the codomain is a Calkin algebra and then using the Busby invariant to pass to the general case. However, there are few a technical points that are worth mentioning. Firstly, the existence of such an approximate units goes back to \cite{le_gall:equivariant-kk} and we modify his construction to suit our needs. Secondly, while considering the case when the codomain is a Calkin algebra, framing the right statement becomes cumbersome due to the fact that neither the multiplier nor the Calkin algebra is a $C_0(G^{(0)})$-algebra. After fleshing out the necessary technical details, one faces the problem of constructing the Busby invariant for groupoid actions. We overcome this but in a slightly less degree of generality than one might hope for. This is mainly due to the fact that the corresponding theory in the case of group actions relies heavily on the automatic continuity theorem of \cite{thomsen:equivariant-kk} and \cite{brown:continuity} which we lack. We conjecture, however, that this should continue to hold in our case as well. We note that Theorem \ref{theo:main-result} is similar to \cite{higson-kasparov:e-theory}*{Theorem 8.1}; however, the precise connection is yet to be examined. While preparing a first draft of this article, we came to know about the paper \cite{gabe:lifting} which provides a generalization of the Choi-Effros lifting theorem. We talk about this paper in relation to our results in the somewhat longer appendix.  

We end this introduction by briefly describing the organization of this article. In Section \ref{sec:groupoids-actions}, we gather the necessary background material on groupoids and their actions in a self-contained  fashion, as much as possible. Section \ref{sec:busby} fleshes out the technicalities of the Busby invariant for groupoid actions. In Section \ref{sec:quasi-central}, we construct the averaging tool, namely, the quasi-invariant, quasi-central, approximate units. Section \ref{sec:main-results} contains the main results. As mentioned above, the somewhat longer Appendix \ref{sec:appendix} is dedicated to the clarification of one of our assumptions and on describing the precise relation of our results to \cite{gabe:lifting}.

\section*{Acknowledgments} The first author was supported by the NFR project 300837 ``Quantum Symmetry'' and the Charles University PRIMUS grant \textit{Spectral Noncommutative Geometry of Quantum Flag Manifolds} PRIMUS/21/SCI/026. He expresses his sincere gratitude to Sergey Neshveyev for several uplifting discussions on the subject and his support throughout the preparation of this article. He also thanks Valerio Proietti for pointing out the similarity of Theorem \ref{theo:main-result} to \cite{higson-kasparov:e-theory}*{Theorem 8.1} and several fruitful discussions, Joachim Zacharias for pointing out \cite{forough-gardella-thomsen:lifting}, Pierre-Yves Le Gall and Jonathan Crespo for several email correspondences. The second author thanks the University of Oslo for support during a visit in summer of 2023, when part of this work was carried out.

\section*{Notations and Conventions} The norm-closed linear span of a subset $S$ of a Banach space is denoted by $[S]$. For a right Hilbert $A$-module $E$, $\mathcal{L}(E)$ denotes the $C^*$-algebra of adjointable operators on $E$. A $*$-homomorphism $\pi : A \rightarrow \mathcal{L}(E)$ is called nondegenerate if $[\pi(A)E]=E$.

\section{Groupoids and their actions}\label{sec:groupoids-actions} In this section, we collect a few basic definitions concerning (locally compact, Hausdorff) groupoids and their actions. We will give precise references as we proceed but we will generally follow \cite{le_gall:equivariant-kk}. To fix notation, we recall that a groupoid is a small category in which each arrow is invertible. If $G$ is a groupoid, we write $G^{(0)}$ for the set of objects; $G^{(2)}$ is the set of composable arrows; $r$ and $s$ are the range and source maps, respectively; $u$ denotes the
unit map from $G^{(0)}$ to $G$. For more details, we refer to \cite{renault:book-groupoid-approach}*{Chapter 1}. In what follows, by a space we mean a locally compact, Hausdorff space and by a groupoid we mean a locally compact, Hausdorff groupoid.

Given a groupoid $G$, a space $X$ and a continuous map $p : X \rightarrow G^{(0)}$, we write, by a slight abuse of notation,
$s^{*}X$ for the pullback $G \tensor[^{s}]{\times}{_{G^{(0)}}} X$ (borrowed from \cite{proietti-yamashita:homology-1}) of the diagram 
\[
\begin{tikzcd}
& X \arrow[d,"p"]\\
G \arrow[r, "s"] & G^{(0)},
\end{tikzcd}    
\]i.e.,
\[
s^{*}X =\{(g,x) \in G \times X \mid s(g)=p(x)\},
\]
and equip it with the relative topology of $G \times X$. Similarly, we shall write $r^{*}X$ for $G \tensor[^{r}]{\times}{_{G^{(0)}}} X$.

\begin{definition}
Let $G$ be a groupoid. A (left) $G$-space is a triple $(X,p,\alpha)$ where $X$ is a space, $p$ is a continuous map from $X$ to $G^{(0)}$, called the \emph{moment map}, and $\alpha$ is a continuous map from
$s^{*}X \rightarrow X$, $\alpha(g,x):=g \cdot x$, satisfying the following conditions.
\begin{itemize}
\item For all $x \in X$, $u_{p(x)} \cdot x=x$.
\item If $(g,h) \in G^{(2)}$ and $(h,x) \in s^{*}X$ then $(g, h \cdot x)\in s^{*}X$ and $(gh) \cdot x=g \cdot (h \cdot x)$. 
\end{itemize}
\end{definition}

It follows that $p(g \cdot x)=r(g)$, i.e., $(g,g \cdot x) \in r^{*}X$. Thus one may view $\alpha$ as a continuous map from $s^{*}X \rightarrow r^{*}X$. To define $G$-actions on $C^{*}$-algebras, we need some preparation. Let $X$ be a space. 

\begin{definition}[{\cite{kasparov:equivariant-kk}*{Definition 1.5}}]
A $C_{0}(X)$-\emph{algebra} is a pair $(A,\theta_{A})$ consisting of a
$C^{*}$-algebra $A$ and a nondegenerate $*$-homomorphism $\theta_{A} :
C_{0}(X) \rightarrow \mathcal{M}(A)$ such that $\theta_{A}(C_{0}(X)) \subset \mathcal{Z} \mathcal{M}(A)$, the center of the multiplier algebra $\mathcal{M}(A)$. 
\end{definition} 

By slightly abusing notation, we will only write $A$ for a $C_{0}(X)$-algebra $(A,\theta_{A})$, suppressing the structural map $\theta_{A}$. We will also write $fa$ for $\theta_{A}(f)(a)$, where $f \in C_{0}(X)$ and $a \in A$. 

Let $A$ be a $C_{0}(X)$-algebra. For $x \in X$, we denote by $\mathfrak{m}_{x}$ the maximal ideal in $C_{0}(X)$ consisting of continuous functions on $X$ that vanish at $x$. The norm-closed linear span $[\mathfrak{m}_{x}A]$ is a (closed, two-sided) ideal in $A$. The \emph{fiber} of $A$ at $x$, denoted by $A(x)$, is the quotient $C^{*}$-algebra $A/[\mathfrak{m}_{x}A]$. We denote the quotient map
by $\pi_{x} : A \rightarrow A(x)$ and write $\pi_{x}(a)=a(x)$, for all $a \in A$. For $f \in C_{0}(X)$ and $a \in A$, one has $(fa)(x)=f(x)a(x)$. For $a \in A$, the map $x \mapsto \left\lVert a(x) \right\rVert$ is \emph{upper semi-continuous} and $\left\lVert a \right\rVert=\sup_{x \in X}\left\lVert a(x) \right\rVert$. If the map $x \mapsto \left\lVert a(x) \right\rVert$ is continuous for all $a \in A$, we say that $A$ is a \emph{continuous field} of $C^{*}$-algebras.   


\begin{definition}
Let $A$ and $B$ be two $C_0(X)$-algebras. A $*$-homomorphism (or a completely positive map) $\phi : A \to B$ is $C_0(X)$-\emph{linear} if $\phi(fa)=f\phi(a)$, for all $f \in C_0(X)$ and $a \in A$.
\end{definition}

If $\phi : A \to B$ is a $C_0(X)$-linear, $*$-homomorphism, then it induces a $*$-homomorphism $\phi_x : A(x) \to B(x)$ for every $x \in X$. By \cite{le_gall:equivariant-kk}*{Proposition 3.1}, $\phi$ is injective (respectively, surjective) if and only if $\phi_x$ is injective (respectively, surjective), for all $x \in X$.  Similarly, if $\phi : A \to B$ is a $C_0(X)$-linear, completely positive map, then by \cite{borys:furstenberg}*{Lemma 3.4}, then it induces a completely positive map  $\phi_x : A(x) \to B(x)$ for every $x \in X$.

\begin{definition}
Let $A$ be a $C_{0}(X)$-algebra.
\begin{itemize}
\item Given an open subset $U$ of $X$, the
\emph{restriction} of $A$ to $U$, denoted by $A_{U}$, is the $C_{0}(U)$-algebra $C_{0}(U)A$.
\item Similarly, given a closed subset $F$ of $X$, the \emph{restriction} of $A$ to $F$, denoted by $A_{F}$, is the $C_{0}(F)$-algebra $A/[m_{F}A]$, where
\[
m_{F}=\{f \in C_{0}(X) \mid f(x)=0 \text{ for all } x \in F\}.  
\] 
\end{itemize}
\end{definition}

Let $A$ be a $C_0(X)$-algebra, $B$ a $C_0(Y)$-algebra and consider the $*$-homomorphism
$$  
C_{0}(X) \otimes C_{0}(Y) \to \mathcal{Z}\mathcal{M}(A) \otimes
\mathcal{Z}\mathcal{M}(B) \rightarrow \mathcal{Z}\mathcal{M}(A
\otimes^{\mathrm{max}} B). $$By \cite{blanchard:deformations}*{Corollary 3.16}, $A \otimes^{\mathrm{max}} B$ together with the above $*$-homomorphism becomes a $C_{0}(X \times Y)$-algebra. Moreover, for $(x,y) \in X \times Y$, the fiber $(A \otimes^{\mathrm{max}} B)(x,y)$ can naturally be identified with $A(x) \otimes^{\mathrm{max}} B(y)$. Strictly speaking, \cite{blanchard:deformations}*{Corollary 3.16} holds only if $X$ is compact. However, as noted in
\cite{echterhoff-williams:c_0-actions}, the same proof holds when the space is $X$ is merely locally compact; see the discussion before Definition 2.3 of \cite{echterhoff-williams:c_0-actions}. 

\begin{definition}
Let $A$ and $B$ be two $C_{0}(X)$-algebras. The \emph{balanced  or relative tensor product}, denoted by $A \otimes^{\mathrm{max}}_{C_{0}(X)} B$, is the $C_{0}(X)$-algebra $(A \otimes^{\mathrm{max}} B)_{\Delta(X)}$, obtained from restricting the $C_{0}(X \times X)$-algebra $A \otimes^{\mathrm{max}} B$ to the diagonal \[\Delta(X)=\{(x,x) \in X \times X \mid x \in X\}.\]  
\end{definition}

Instead of the maximal tensor product, we could have as well chosen the minimal tensor product $\otimes^{\mathrm{min}}$ in the above definition to define $\otimes^{\mathrm{min}}_{C_{0}(X)}$. However, $\otimes^{\mathrm{min}}_{C_{0}(X)}$ is not associative; see \cite{blanchard:deformations}*{Section 3.3}. There are at least two more different ways of defining $\otimes^{\mathrm{max}}_{C_{0}(X)}$ (and of course, $\otimes^{\mathrm{min}}_{C_{0}(X)}$); see \cite{kasparov:equivariant-kk}*{Definition 1.6} and \cite{echterhoff-williams:c_0-actions}*{Definition 2.5}. The latter considers the (closed, two-sided) ideal \[I:=\{ fa \otimes^{\mathrm{max}} b - a \otimes^{\mathrm{max}}fb \mid f \in C_{0}(X), a \in A, b \in B\}\]of $A \otimes^{\mathrm{max}} B$ and defines $A \otimes^{\mathrm{max}}_{C_{0}(X)} B$ to be the quotient $A \otimes^{\mathrm{max}} B/I$. Using the relative tensor product, we can now define pullbacks.

\begin{definition}
Let $A$ be a $C_{0}(X)$-algebra, $Y$ be a locally compact, Hausdorff space, and $p : Y \rightarrow X$ be a continuous map. The \emph{pullback}, denoted by $p^{*}A$, is the $C_{0}(Y)$-algebra $A \otimes_{C_{0}(X)} C_{0}(Y)$.
\end{definition}

Note that, $p^*A$ is, a priori, only a $C_0(X)$-algebra; the $C_0(Y)$-algebra structure comes from the canonical inclusion $C_0(Y) \rightarrow \mathcal{M}(A \otimes_{C_0(X)} C_0(Y))$. Note further, that it is not necessary to specify whether we use the minimal tensor product $\otimes^{\mathrm{min}}$ or the maximal tensor product $\otimes^{\mathrm{max}}$ because $C_{0}(Y)$ is nuclear. Alternatively, one can define $p^{*}A$ as the restriction of the $C_{0}(X \times Y)$-algebra $A \otimes C_{0}(Y)$ to the closed set $\{(p(y),y) \mid y \in Y\}$ of $X \times Y$. For $y \in Y$, $(p^{*}A)(y)$ is naturally identified with $A(p(y))$. Given two $C_{0}(X)$-algebras $A$ and $B$, one can identify $p^{*}A \otimes^{\mathrm{max}}_{C_{0}(Y)} p^{*}B$ with $p^{*}(A \otimes^{\mathrm{max}}_{C_{0}(X)} B)$; similar identification holds if we replace $\otimes^{\mathrm{max}}$ with $\otimes^{\mathrm{min}}$. Given a $C_{0}(X)$-linear, $*$-homomorphism (or a completely positive map) $\phi : A \rightarrow B$, there is a $C_{0}(Y)$-linear, $*$-homomorphism (respectively, a completely positive map) $p^{*}\phi : p^{*}A \rightarrow p^{*}B$. Finally, by \cite{popescu:equivariant-e-theory}*{Corollary 1.26}, the pullback functor is \emph{exact}, i.e., if  
\[
0 \rightarrow I \rightarrow A \rightarrow B \rightarrow 0
\]is a short exact sequence of $C_{0}(X)$-algebras and $C_{0}(X)$-linear, $*$-homomorphisms, then
\[
0 \rightarrow p^{*}I \rightarrow p^{*}A \rightarrow p^{*}B \rightarrow 0,
\]too is a short exact sequence of $C_{0}(Y)$-algebras and $C_{0}(Y)$-linear $*$-homomorphisms, for every continuous map $p : Y \rightarrow X$. This is will be important for us. In particular, this implies that the ``fiber'' functor, i.e., that sends $A$ to $A(x)$, being the pullback functor via ${x} \rightarrow X$, is exact.

Before going to the definition of a groupoid action on a $C^{*}$-algebra, let us mention that oftentimes, it is helpful to view a $C_{0}(X)$-algebra $A$ as an (upper semi-continuous) bundle of $C^{*}$-algebras. For that purpose, we set $\A=\bigsqcup_{x \in X} A(x)$ and let $\pi : \A \to X$ to be the projection map. We topologize $\mathcal{A}$ by taking the subsets 
\begin{equation}
W(a, U, \epsilon)=\{b \in \A \mid \pi(b) \in U \text{ and } \left\lVert a(\pi(b))-b  \right\rVert < \epsilon \}, 
\end{equation}
where $a \in A$, $U \subset X$ is open and $\epsilon >0$, as a basis. Let $\Gamma_0(X;\A)$ denote the algebra of continuous sections of the upper semi-continuous $C^{*}$-bundle $\A$ vanishing at infinity, then $A$ and $\Gamma_0(X;\A)$ are isomorphic as $C_{0}(X)$-algebras. Moreover, this process can be reversed; see, for example,
\cite{kirchberg-wassermann:continuous-bundles}*{Introduction}, \cite{rieffel:continuous-fields}*{Section 1}, or \cite{williams:book-crossed-products}*{Appendix C} for more details. As an example, let $A$ be a $C_0(X)$-algebra and let $\pi : \A \to X$ be the upper semi-continuous $C^*$-bundle associated to $A$. If $p : Y \to X$ is a continuous map, one defines the pullback $p^*\A=\{(c, y) \in \A \times Y \mid \pi(c)=p(y)\}$, which is again an upper semi-continuous $C^*$-bundle. It follows from \cite{raeburn-williams:pullbacks}*{Proposition 1.3} that $p^*A$ and $\Gamma_{0}(X;p^{*} \mathcal{A})$ are isomorphic as $C_{0}(Y)$-algebras.  

Returning back to defining groupoid actions on $C^{*}$-algebras, suppose now that we have a $C_{0}(G^{(0)})$-algebra $A$ and a $C_{0}(G)$-linear, $*$-isomorphism $\alpha : s^{*}A \rightarrow r^{*}A$. Then $\alpha$ induces $*$-isomorphisms $\alpha_{g} : A(s(g)) \rightarrow A(r(g))$ for each $g \in G$.

\begin{definition}[{\cite{le_gall:equivariant-kk}*{Definition 3.5}}]
Let $G$ be a groupoid. A $G$-$C^*$-\emph{algebra} is a triple $(A,\theta_{A},\alpha)$ where $(A,\theta_{A})$ is a $C_{0}(G^{(0)})$-algebra and $\alpha$ is a $C_0(G)$-linear, $*$-isomorphism $\alpha : s^*A \to r^*A$ such that $\alpha_g \circ \alpha_h=\alpha_{gh}$, whenever $(g,h) \in G^{(2)}$.
\end{definition}

Again by slightly abusing notation, we will only write $(A,\alpha)$ for a $G$-$C^*$-algebra $(A,\theta_{A},\alpha)$, suppressing the structural map $\theta_{A}$. The following lemma puts the above definition in terms of the bundle picture.

\begin{lemma}[{\cite{muhly-williams:renault-theorem}*{Lemma 4.3}}]  \label{lem:continuity} 
Let $G$ be a groupoid and $(A,\theta_{A})$ be a $C_0(G^{(0)})$-algebra. Suppose that $(\alpha_g)_{g \in G}$ is a family of $*$-isomorphisms $\alpha_g : A(s(g)) \to A(r(g))$ such that $\alpha_g \circ \alpha_h=\alpha_{gh}$, whenever $(g,h) \in
G^{(2)}$. Then the following are equivalent.
\begin{itemize}
\item There exists a $C_0(G)$-linear, $*$-isomorphism $\alpha : s^*A \to r^*A$ such that $\alpha(s^*(a))(g)=\alpha_g(a(s(g)))$, making $(A,\theta_{A},\alpha)$ into a $G$-$C^*$-algebra.  
\item For $(g,a) \in s^{*} \mathcal{A}$, setting $\tilde{\alpha}(g,a)=\alpha_g(a)$ makes $(\mathcal{A},\pi, \tilde{\alpha})$ into a $G$-space.
\item For any $a \in A$, the map $g \mapsto \alpha_g(a(s(g)))$ is a bounded continuous section of $G \to r^*\mathcal{A}$.
\end{itemize}
\end{lemma}

We end this section with one more definition.

\begin{definition}
Let $G$ be a groupoid, and $(A,\alpha)$ and $(B,\beta)$ be two $G$-algebras. A $C_0(G^{(0)})$-linear map $\phi \colon A \to B$ is said to be $G$-\emph{equivariant} if $(r^*\phi) \circ \alpha= \beta \circ (s^*\phi)$, i.e., for all $g \in G$, $\phi_{r(g)} \circ \alpha_g=\beta_g \circ \phi_{s(g)}$.
\end{definition}

\section{Equivariant extensions and the Busby invariant}\label{sec:busby} This section describes how for a groupoid $G$, extensions of $G$-$C^*$-algebras are in one-one correspondence with their Busby invariants. This extends well-known results from \cite{thomsen:equivariant-kk}*{Section 2} and \cite{kasparov:operator-k}*{Section 7} in the case when $G$ is a group and from \cite{blanchard:subtriviality}*{Section 1} for $C_0(X)$-algebras with $X$ compact. Most of the results are straightforward; however, for the readers' convenience, we provide full details. To begin with, we will throughout this article write $q_A : \mathcal{M}(A) \rightarrow \mathcal{Q}(A)$ for the canonical projection. Let $X$ be a space and let $A$ be a separable, $C_0(X)$-algebra. For any $x \in X$, we define 
\begin{equation}\label{eq:xi}
\xi_{x} \colon \mathcal{M}(A) \to \mathcal{M}(A(x)), \quad \xi_x(m)(e(x))=(me)(x),
\end{equation}
for any $m \in \mathcal{M}(A)$ and $e \in A$, i.e., $\xi_x$ is the canonical surjective extension of $\pi_x : A \to A(x)$, see, for instance, \cite{akemann-pedersen-tomiyama:multipliers}*{Theorem 4.2}. In turn, $\xi_x$ induces
\begin{equation}\label{eq:zeta}
\zeta_x : \mathcal{Q}(A) \to \mathcal{Q}(A(x))  
\end{equation}
such that the following diagram 
\begin{equation}\label{eq:xi-zeta-square}
\begin{tikzcd}
\mathcal{M}(A) \arrow[d,"\xi_x"] \arrow[r,"q_A"] & \mathcal{Q}(A) \arrow[d,"\zeta_x"] \\
\mathcal{M}(A(x)) \arrow[r,"q_{A(x)}"] & \mathcal{Q}(A(x))
\end{tikzcd}
\end{equation}     
commutes. It is an easy check that if $f \in \mathfrak{m}_x$ and $m \in \mathcal{M}(A)$, then
\begin{equation}\label{eq:xi-maximal-ideal}
\xi_x(\theta_A(f)m)=\zeta_x(q_A(\theta_A(f)m))=0.  
\end{equation}
Given a $C_0(X)$-extension of $A$ by another $C_0(X)$-algebra $I$, i.e., a short exact sequence 
\begin{equation}\label{eq:extension}\tag{$\eta$}
0 \rightarrow I \xrightarrow{\iota} B \xrightarrow{p} A \rightarrow 0,
\end{equation}
of $C_0(X)$-algebras and $C_0(X)$-linear, $*$-homomorphisms $\iota : I \rightarrow B$ and $p : B \rightarrow A$, for any $x \in X$, the sequence 
\begin{equation}\label{eq:extension-fiber}\tag{$\eta(x)$}
0 \rightarrow I(x) \xrightarrow{\iota_x} B(x) \xrightarrow{p_x} A(x) \rightarrow 0,  
\end{equation}
is exact; conversely, given \eqref{eq:extension} such that for each $x \in X$, \eqref{eq:extension-fiber} is exact then \eqref{eq:extension} is exact. This follows from the fact that the ``fiber'' functor, i.e., that sends $A$ to $A(x)$ is the pullback functor via ${x} \rightarrow X$, which is exact and \cite{popescu:equivariant-e-theory}*{Lemma 1.25}. For such an extension, we define 
\begin{equation}
\bar{\mu}(\eta) : B \to \mathcal{M}(I), \quad \bar{\mu}(\eta)(b)(e)=b \iota(e)
\end{equation}
for all $b \in B$ and $e \in I$. Note that we also have the $*$-homomorphism 
\begin{equation}\label{eq:mu-bar-fiber}
\bar\mu(\eta(x)) : B(x) \rightarrow \mathcal{M}(I(x))
\end{equation}
for the extension \eqref{eq:extension-fiber}. We observe that although $\mathcal{M}(I)$ is not a $C_0(X)$-algebra (see, for instance, \cite{williams:book-crossed-products}*{Remark C.14}), the $*$-homomorphism $\bar \mu$ satisfies a type of ``$C_0(X)$-linearity" in the sense of the following definition.

\begin{definition}\label{def:linearity-mu-bar}
Let $B$ and $I$ be $C_0(X)$-algebras. A $*$-homomorphism (or a completely positive map) $\phi : B \to \mathcal{M}(I)$ is said to be $C_0(X)$-linear if 
\[
\phi(\theta_A(f)b)=\theta_I(f)\phi(b),
\]for all $f \in C_0(X)$ and $b \in B$. 
\end{definition}

If $I$ is separable, then given a $C_0(X)$-linear $\phi : B \rightarrow \mathcal{M}(I)$, $b \in B$, $x \in X$ and $f \in \mathfrak{m}_x$, by \eqref{eq:xi-maximal-ideal}, it follows that
\[
\xi_x(\phi(\theta_B(f)b))=\xi_x(\theta_I(f)\phi(b))=0,
\]so that
\begin{equation}\label{eq:phi-fiber}
\phi_x : B(x) \to \mathcal{M}(I(x)), \quad \phi_x(b(x))=\xi_x(\phi(b))
\end{equation} 
is a well-defined $*$-homomorphism such that the following diagram
\begin{equation}\label{eq:phi-fiber-square}
\begin{tikzcd}
B \arrow[d,"\pi_x"] \arrow[r,"\phi"] & \mathcal{M}(I) \arrow[d,"\xi_x"] \\
B(x) \arrow[r,"\phi_x"] & \mathcal{M}(I(x))
\end{tikzcd}
\end{equation}
commutes. 

\begin{lemma}
The $*$-homomorphism 
\[
\bar\mu(\eta(x)) : B(x) \rightarrow \mathcal{M}(I(x))
\]
for the extension \eqref{eq:extension-fiber} coincides with the $*$-homomorphism  
\[
\bar\mu(\eta)_x : B(x) \rightarrow \mathcal{M}(I(x))
\]
as defined in \eqref{eq:phi-fiber} above, whenever $I$ is separable.
\end{lemma}
  
\begin{proof}
The proof is a plain computation which goes as follows. We fix $x \in X$, $b \in B$ and $e \in I$. We then have 
\allowdisplaybreaks{
\[
\begin{aligned}
\bar\mu(\eta)_x(\pi_x(b))(e(x)) &{}=\xi_x(\mu(\eta)(b))(e(x))\\
&{}=((\bar\mu(\eta)(b))(e))(x)\\
&{}=(b\iota(e))(x)\\
&{}=b(x)\iota_x(e(x))\\
&{}=\bar\mu(\eta(x))(b(x))e(x);
\end{aligned}  
\]
}where the first equality is by commutativity of \eqref{eq:phi-fiber-square}; the second equality is by definition of $\xi_x$; the third equality is by definition of $\bar\mu(\eta)$; and the fifth equality is by definition of $\bar\mu(\eta(x))$. Since $\pi_x$ is surjective, the proof is complete.
\end{proof}

In particular, the following square
\begin{equation}\label{eq:mu-bar-fiber-square}
\begin{tikzcd}
B \arrow[d,"\pi_x"] \arrow[r,"\bar \mu(\eta)"] & \mathcal{M}(I) \arrow[d,"\xi_x"]\\
B(x) \arrow[r,"\bar \mu(\eta(x))"] & \mathcal{M}(I(x))
\end{tikzcd}
\end{equation}
commutes. As in the ordinary $C^*$-algebraic case, we define the \emph{Busby invariant} 
\begin{equation}\label{eq:busby}
\mu(\eta) : A \to \mathcal{Q}(I), \quad \mu(\eta)(p(b))=q_I(\bar{\mu}(\eta)(b)). 
\end{equation}
Note that we also have the Busby invariant 
\begin{equation}\label{eq:busby-fiber}
\mu(\eta(x)) : A(x) \rightarrow \mathcal{Q}(I(x))
\end{equation}
for the extension \eqref{eq:extension-fiber}. Again although $\mathcal{Q}(I)$ is not a $C_0(X)$-algebra, the Busby invariant $\mu$ satisfies a type of ``$C_0(X)$-linearity" in the sense of the following definition.

\begin{definition}\label{def:linearity-Busby}
Let $A$ and $I$ be $C_0(X)$-algebras. A $*$-homomorphism (or a completely positive map) $\psi : A \to \mathcal{Q}(I)$ is said to be $C_0(X)$-linear if 
\[
\psi(\theta_A(f)a)=q_I\theta_I(f)\psi(a),
\]for all $f \in C_0(X)$ and $a \in A$. 
\end{definition}

Once again, if $I$ is separable, given such a $C_0(X)$-linear $\psi : A \rightarrow \mathcal{Q}(I)$, $a \in A$, $x \in X$ and $f \in \mathfrak{m}_x$, by \eqref{eq:xi-maximal-ideal}, it follows that
\[
\zeta_x(\psi(\theta_A(f)a))=\zeta_x(q_I(\theta_I(f)))\zeta_x(\psi(a))=0,
\]so that
\begin{equation}\label{eq:psi-fiber}
\psi_x : A(x) \to \mathcal{Q}(I(x)), \quad \psi_x(a(x))=\zeta_x(\psi(a))
\end{equation} 
is a well-defined $*$-homomorphism such that the following diagram
\begin{equation}\label{eq:psi-fiber-square}
\begin{tikzcd}
A \arrow[d,"\pi_x"] \arrow[r,"\psi"] & \mathcal{Q}(I) \arrow[d,"\zeta_x"] \\
A(x) \arrow[r,"\psi_x"] & \mathcal{Q}(I(x))
\end{tikzcd}
\end{equation}
commutes.

\begin{lemma}
The Busby invariant  
\[
\mu(\eta(x)) : A(x) \rightarrow \mathcal{Q}(I(x))
\]
for the extension \eqref{eq:extension-fiber} coincides with the $*$-homomorphism  
\[
\mu(\eta)_x : A(x) \rightarrow \mathcal{Q}(I(x))
\]
as defined in \eqref{eq:psi-fiber} above, whenever $I$ is separable.
\end{lemma}

\begin{proof}
The proof is a plain computation which goes as follows. We fix $x \in X$ and $b \in B$. We then have 
\allowdisplaybreaks{
\[
\begin{aligned}
\mu(\eta)_x(\pi_x(p(b))) &{}=\zeta_x(\mu(\eta)(p(b)))\\
&{}=\zeta_x(q_I(\bar \mu(\eta)(b)))\\
&{}=q_{I(x)}\xi_x(\bar \mu(\eta)(b))\\
&{}=q_{I(x)}\bar \mu(\eta(x))(\pi_x(b))\\
&{}=\mu(\eta(x))(p_x\pi_x(b))\\
&{}=\mu(\eta(x))(\pi_xp(b));
\end{aligned}  
\]
}where the first equality is by commutativity of \eqref{eq:psi-fiber-square}; the second equality is by definition of $\mu(\eta)$; the third equality is by commutativity of \eqref{eq:xi-zeta-square}; the fourth equality is by commutativity of \eqref{eq:mu-bar-fiber-square}; the fifth equality is by definition of $\mu(\eta(x))$ and the sixth equality is because $p$ is $C_0(X)$-linear. Since $p$ and $\pi_x$ are surjective, the proof is complete.
\end{proof}

In particular, the following square
\begin{equation}\label{eq:mu-fiber-square}
\begin{tikzcd}
A \arrow[d,"\pi_x"] \arrow[r,"\mu(\eta)"] & \mathcal{Q}(I) \arrow[d,"\zeta_x"]\\
A(x) \arrow[r,"\mu(\eta(x))"] & \mathcal{Q}(I(x))
\end{tikzcd}
\end{equation}
commutes.

\begin{proposition}\label{prop:c0x-busby}
Let $I$ and $A$ be two separable, $C_0(X)$-algebras and let $\mu : A \to \mathcal{Q}(I)$ be a $C_0(X)$-linear, $*$-homomorphism. Set 
\[
B=\{(m,a) \in \mathcal{M}(I) \oplus A \mid q_I(m)=\mu(a) \},
\]
\[
\iota : I \rightarrow B, e \mapsto (e,0), \quad p : B \rightarrow A, (m,a) \mapsto a.
\]
Then the following statements hold.
\begin{enumerate}
  \item $B$ can be uniquely made into a $C_0(X)$-algebra such that $\iota$ and $p$ are $C_0(X)$-linear.
  \item For any $x \in X$, the fiber $B(x)$ is isomorphic to 
  \[B_x:=\{(m', a') \in \mathcal{M}(I(x)) \oplus A(x) \mid q_{I(x)}(m')=\mu_x(a') \}.
  \]
  \item The extension 
  \begin{equation}\tag{$\eta$}\label{eq:extension-1}
    0 \rightarrow I \xrightarrow{\iota} B \xrightarrow{p} A \rightarrow 0
  \end{equation}is a $C_0(X)$-extension such that $\mu(\eta)=\mu$.
\end{enumerate}
\end{proposition}

\begin{proof}
We begin by observing that since $\mu$ is $C_0(X)$-linear, for any $f \in C_0(X)$, and $(m,a) \in B$, $(\theta_I(f)m,\theta_A(f)a)$ is in $B$ too. Let then $\theta_B : C_0(X) \rightarrow \mathcal{M}(B)$ to be the $*$-homomorphism defined as $\theta_B(f)(m,a)=(\theta_I(f)m,\theta_A(f)a)$, for $f \in C_0(X)$, and $(m,a) \in B$. It is easy to see that $\theta_B(C_0(X)) \subset \mathcal{ZM}(B)$. It is equally easy to see that $\iota$ and $p$ are $C_0(X)$-linear. Next, we show that $\theta_B$ is nondegenerate. For that, we set $B_X=[C_0(X)B]$. By $C_0(X)$-linearity of $\iota$ and $p$ together with exactness of \eqref{eq:extension-1}, we obtain that 
\[
0 \rightarrow I \xrightarrow{\iota} B_X \xrightarrow{p} A \rightarrow 0, 
\]is also exact. Thus we have the following commutative diagram
\[
\begin{tikzcd}
0 \arrow[r] & I \arrow[r,"\iota"] \arrow[d,sloped,"="] & B_X \arrow[r,"p"] \arrow[d,sloped,"\subseteq"] & A \arrow[r] \arrow[d,sloped,"="] & 0\\
0 \arrow[r] & I \arrow[r,"\iota"] & B \arrow[r,"p"] & A \arrow[r] & 0,
\end{tikzcd}
\]with exact rows. The Five lemma then implies that $B=B_X$, i.e., $B$ is a $C_0(X)$-algebra. To conclude (1), we now show uniqueness of this $C_0(X)$-algebra structure. For that, let $\theta'_B : C_0(X) \rightarrow \mathcal{ZM}(B)$ be another nondegenerate, $*$-homomorphism which makes $B$ into a $C_0(X)$-algebra such that $\iota$ and $p$ are $C_0(X)$-linear. We fix $f \in C_0(X)$, $(m,a) \in B$, and write $\theta'_B(f)(m,a)=(m',a')$. Let $(u_n)_{n \in \mathbb{N}}$ be an approximate unit of $I$. Then we note that $(m'u_n,0)=\theta'_B(f)(m,a)(u_n,0)=(m,a)\theta'_B(f)(u_n,0)=(m,a)(\theta_I(f)u_n,0)=(\theta_I(f)mu_n,0)$. Letting $n \rightarrow \infty$, we see that $m'=\theta_I(f)m$; here we have used the centrality of $\theta_B$ and $\theta'_B$ together with the $C_0(X)$-linearity of $\iota$. Now since $p$ is also $C_0(X)$-linear, we obtain that $a'=\theta_A(f)a$, by applying $p$ to the both sides of $\theta'_B(f)(m,a)=(m',a')$. Thus $\theta_B(f)(m,a)=(\theta_I(f)m,\theta_A(f)a)$ and this concludes (1). Moreover, (3) follows as well since \eqref{eq:extension-1} is in particular an ordinary $C^*$-extension.

For (2), we fix $x \in X$. We observe that if $(m,a)$ is in $B$, then by commutativity of \eqref{eq:xi-zeta-square} and \eqref{eq:psi-fiber-square}, $(\xi_x(m),a(x))$ is in $B_x$. We set $\iota'_x : I(x) \rightarrow B_x$, $e(x) \mapsto (e(x),0)$, $p'_x : B_x \rightarrow A(x)$, $(m',a') \mapsto a'$ and $\phi : B \rightarrow B_x$, $(m,a) \mapsto (\xi_x(m),a(x))$. Then, by definition, we have the following commutative diagram
\[
\begin{tikzcd}
0 \arrow[r] & I \arrow[r,"\iota"] \arrow[d,"\pi_x"] & B \arrow[r,"p"] \arrow[d,"\phi"] & A \arrow[r] \arrow[d,"\pi_x"] & 0\\
0 \arrow[r] & I(x) \arrow[r,"\iota'_x"] & B_x \arrow[r,"p'_x"] & A(x) \arrow[r] & 0,
\end{tikzcd}
\]with exact rows. The Snake lemma then implies that the following sequence
\[
0 \rightarrow [\mathfrak{m}_xI] \rightarrow \ker(\phi) \rightarrow [\mathfrak{m}_xA] \rightarrow 0  
\]
is exact. By \eqref{eq:xi-maximal-ideal}, we observe that $[\mathfrak{m}_xB] \subseteq \ker(\phi)$. Moreover, we have the following commutative diagram
\[
\begin{tikzcd}
0 \arrow[r] & {}[\mathfrak{m}_xI] \arrow[r] \arrow[d,sloped,"="] & {}[\mathfrak{m}_xB] \arrow[r] \arrow[d,sloped,"\subseteq"] & {}[\mathfrak{m}_xA] \arrow[r] \arrow[d,sloped,"="] & 0\\
0 \arrow[r] & {}[\mathfrak{m}_xI] \arrow[r] & \ker(\phi) \arrow[r] & {}[\mathfrak{m}_xA] \arrow[r] & 0,
\end{tikzcd}
\]with exact rows. Again by the five lemma, we conclude that $\ker(\phi)=[\mathfrak{m}_xB]$, i.e., $\phi$ descends to an isomorphism from $B(x)$ to $B_x$. This concludes (2) and the proof.
\end{proof}



We now proceed towards a $G$-equivariant version of Proposition \ref{prop:c0x-busby} and we begin by fixing some notations. Let $G$ be a second countable groupoid and $(A,\alpha)$ be a separable, $G$-$C^*$-algebra. We will write 
\[
\bar \alpha : \mathcal{M}(s^*A) \rightarrow \mathcal{M}(r^*A), \quad \tilde \alpha : \mathcal{Q}(s^*A) \rightarrow \mathcal{Q}(r^*A),   
\]for the induced $*$-isomorphisms, so that the diagram
\begin{equation}
\begin{tikzcd}
\mathcal{M}(s^*A) \arrow[d,"q_{s^*A}"] \arrow[r,"\bar \alpha"] & \mathcal{M}(r^*A) \arrow[d,"q_{r^*A}"]\\  
\mathcal{Q}(s^*A) \arrow[r,"\tilde \alpha"] & \mathcal{Q}(r^*A)
\end{tikzcd}  
\end{equation}
commutes. Since $\alpha$ is $C_0(G)$-linear, so are $\bar \alpha$ and $\tilde \alpha$. Thus for each $g \in G$, the following diagrams
\begin{equation}
\begin{tikzcd}
\mathcal{M}(s^*A) \arrow[d,"\xi_g"] \arrow[r,"\bar \alpha"] & \mathcal{M}(r^*A) \arrow[d,"\xi_g"] \\
\mathcal{M}(A(s(g))) \arrow[r,"\bar \alpha_g"] & \mathcal{M}(A(r(g))),
\end{tikzcd}
\begin{tikzcd}
\mathcal{Q}(s^*A) \arrow[d,"\zeta_g"] \arrow[r,"\tilde \alpha"] & \mathcal{Q}(r^*A) \arrow[d,"\zeta_g"] \\
\mathcal{Q}(A(s(g))) \arrow[r,"\tilde \alpha_g"] & \mathcal{Q}(A(r(g))),
\end{tikzcd}  
\end{equation}and
\begin{equation}\label{eq:alphabar-alphatilde}
\begin{tikzcd}
\mathcal{M}(A(s(g))) \arrow[d,"q_{A(s(g))}"] \arrow[r,"\bar \alpha_g"] & \mathcal{M}(A(r(g))) \arrow[d,"q_{A(r(g))}"]\\  
\mathcal{Q}(A(s(g))) \arrow[r,"\tilde \alpha_g"] & \mathcal{Q}(A(r(g)))
\end{tikzcd}  
\end{equation}
commute. An extension 
\begin{equation}\label{eq:extension-2}\tag{$\eta$}
0 \rightarrow I \xrightarrow{\iota} B \xrightarrow{p} A \rightarrow 0  
\end{equation}
of $G$-$C^*$-algebras and $G$-equivariant, $*$-homomorphisms, will be called a $G$-$C^*$-extension. In particular, \eqref{eq:extension-2} is a $C_0(G^{(0)})$-extension and therefore $\bar \mu(\eta)$ is $C_0(G^{(0)})$-linear. As one might expect, $\bar \mu(\eta)$ is $G$-equivariant in a sense, which we make precise now.

\begin{definition}\label{def:mubar-equivariant}
Let $(B,\beta)$ and $(I,\gamma)$ be $G$-$C^*$-algebras. A $C_0(G^{(0)})$-linear $*$-homomorphism (or completely positive map) $\phi : B \rightarrow \mathcal{M}(I)$ is said to be $G$-equivariant if $\phi_{r(g)}\beta_g=\bar \gamma_g \phi_{s(g)}$, i.e., the diagram 
\begin{equation}
\begin{tikzcd}
B_{s(g)} \arrow[d,"\phi_{s(g)}"] \arrow[r,"\beta_g"] & B_{r(g)} \arrow[d,"\phi_{r(g)}"] \\
\mathcal{M}(I(s(g))) \arrow[r,"\bar \gamma_g"] & \mathcal{M}(I(r(g)))
\end{tikzcd}
\end{equation}commutes for all $g \in G$.
\end{definition}

\begin{lemma}\label{lem:mubar-equivariant}
The $*$-homomorphism 
\[
\bar \mu(\eta) : B \rightarrow \mathcal{M}(I)  
\]for the extension \eqref{eq:extension-2} is $G$-equivariant, whenever $I$ is separable.  
\end{lemma}

\begin{proof}
We fix $b \in B$, $e \in I$. We then have for all $g \in G$,
\allowdisplaybreaks{
\[
\begin{aligned}
\bar \mu_{r(g)}(\beta_g(b(s(g))))(e(r(g)))=&{}\beta_g(b(s(g)))\iota_{r(g)}(e(r(g)))\\
=&{}\beta_g(b(s(g)))\iota_{r(g)}(\gamma_g(\gamma_g^{-1}(e(r(g)))))\\ 
=&{}\beta_g(b(s(g)))\beta_g(\iota_{s(g)}(\gamma_g^{-1}(e(r(g)))))\\ 
=&{}\gamma_g(b(s(g))\iota_{s(g)}(\gamma_g^{-1}(e(r(g)))))\\
=&{}\gamma_g(\bar \mu_{s(g)}(b(s(g)))(\gamma_g^{-1}(e(r(g)))))\\ 
=&{}\bar \gamma_g (\bar \mu_{s(g)}(b(s(g))))(e(r(g))),
\end{aligned}
\]
}where the first equality is by definition of $\bar \mu_{r(g)}$; the third equality is by $G$-equivariance of $\iota$; the fifth equality is again by definition of $\bar \mu_{s(g)}$ and the sixth equality is by definition of $\bar \gamma_g$. This completes the proof.
\end{proof}

Similarly, the Busby invariant $\mu : A \rightarrow \mathcal{Q}(I)$ is $C_0(G^{(0)})$-linear. Moreover, it is $G$-equivariant in the following sense.

\begin{definition}\label{def:busby-equivariant}
Let $(A,\alpha)$ and $(I,\gamma)$ be $G$-$C^*$-algebras. A $C_0(G^{(0)})$-linear $*$-homomorphism (or completely positive map) $\psi : A \rightarrow \mathcal{Q}(I)$ is said to be $G$-equivariant if $\psi_{r(g)}\alpha_g=\tilde \gamma_g \psi_{s(g)}$, i.e., the diagram 
\begin{equation}
\begin{tikzcd}
A_{s(g)} \arrow[d,"\psi_{s(g)}"] \arrow[r,"\alpha_g"] & A_{r(g)} \arrow[d,"\psi_{r(g)}"] \\
\mathcal{Q}(I(s(g))) \arrow[r,"\tilde \gamma_g"] & \mathcal{Q}(I(r(g)))
\end{tikzcd}
\end{equation}commutes for all $g \in G$.
\end{definition}

\begin{lemma}\label{lem:busby-equivariant}
The Busby invariant 
\[
\mu : A \rightarrow \mathcal{Q}(I)  
\]for the extension \eqref{eq:extension-2} is $G$-equivariant, whenever $I$ is separable.  
\end{lemma}

\begin{proof}
We fix $b \in B$. We then have for all $g \in G$,
\allowdisplaybreaks{
\[
\begin{aligned}
\mu_{r(g)}(\alpha_g(p_{s(g)}(b_{s(g)})))=&{}\mu_{r(g)}(p_{r(g)}(\beta_g(b_{s(g)})))\\
=&{}q_{I(r(g))}(\bar \mu_{r(g)}(\beta_g(b(s(g)))))\\ 
=&{}q_{I(r(g))}(\bar \gamma_g(\bar \mu_{s(g)}(b(s(g)))))\\ 
=&{}\tilde \gamma_g(q_{I(s(g))}(\bar \mu_{s(g)}(b(s(g)))))\\ 
=&{}\tilde \gamma_g(\mu_{s(g)}(p_{s(g)}(b(s(g))))),
\end{aligned}  
\]  
}where the first equality is by equivariance of $p$; the second equality is by definition of $\mu_{r(g)}$; the third equality is by Lemma \ref{lem:mubar-equivariant}; the fourth equality is by the commutativity of the square \eqref{eq:alphabar-alphatilde} and the fifth equality is again by definition of $\mu_{s(g)}$. This completes the proof. 
\end{proof}

One might expect that the converse to the lemma above holds too, i.e., for a $G$-equivariant, $*$-homomorphism $\mu : A \rightarrow \mathcal{Q}(I)$, the pullback $B$ in Proposition \ref{prop:c0x-busby} is a $G$-$C^*$-algebra and the extension \eqref{eq:extension-1} is a $G$-$C^*$-extension such that $\mu(\eta)=\mu$. Indeed, when $G$ is a group, this is \cite{thomsen:equivariant-kk}*{Theorem 2.2}. However, we note that \cite{thomsen:equivariant-kk}*{Theorem 2.2} is a corollary of the automatic continuity result \cite{thomsen:equivariant-kk}*{Theorem 2.1} which was more generally proved in \cite{brown:continuity}*{Theorem 2}. Lacking such a result, we follow \cite{kasparov:operator-k}*{Section 7} and make the following definition.

\begin{definition}\label{def:g-admissible}
Let $G$ be a second countable groupoid, $(I,\gamma)$ and $(A,\alpha)$ be two separable, $G$-$C^*$-algebras and $\mu : A \rightarrow \mathcal{Q}(I)$ be a $G$-equivariant, $*$-homomorphism. We say $\mu$ is $G$-admissible if given $a \in A$ and $m \in \mathcal{M}(I)$ with $q_{I}(m)=\mu(a)$, the section $g \mapsto \gamma_g\xi_{s(g)}(m)\gamma_g^{-1}$ is continuous.
\end{definition}

Some remarks are in order. Firstly, we topologize $\bigsqcup_{g \in G} \mathcal{M}(I(r(g)))$ as in \cite{akemann-pedersen-tomiyama:multipliers}*{Section 3} and note that although the reference deals with continuous fields, the topology nevertheless makes sense with upper semi-continuity instead of continuity; see also \cite{williams:book-crossed-products}*{Lemma C.11}. Secondly, with this topology, the section $g \mapsto \gamma_g\xi_{s(g)}(m)\gamma_g^{-1}$ in the above definition is required to be norm-continuous instead of strictly-continuous. Finally, that the class of $G$-admissible, $*$-homomorphisms is not empty follows from the following corollary to Lemma \ref{lem:mubar-equivariant}.

\begin{corollary}\label{cor:busby-g-admissible}
The Busby invariant  
\[
\mu(\eta) : A \rightarrow \mathcal{Q}(I)
\]
for the extension \eqref{eq:extension-2} is $G$-admissible, whenever $B$ is separable. 
\end{corollary}

\begin{proof}
We fix $a \in A$ and $m \in \mathcal{M}(I)$ with $q_I(m)=\mu(a)$. We choose $b \in B$ such that $m=\bar \mu(\eta)(b)$. By Lemma \ref{lem:mubar-equivariant}, we conclude that the section $g \mapsto \gamma_g\xi_{s(g)}(m)\gamma_g^{-1}$ is the image of the section $g \mapsto \beta_g(b(s(g))$ under $\bar \mu(\eta)$. The result now follows since the latter section is continuous and $\bar \mu(\eta)$ is a $*$-homomorphism.
\end{proof}


The following corollary to Proposition \ref{prop:c0x-busby} sets up the bijection between $G$-$C^*$-extensions and $G$-admissible, $*$-homomorphisms as defined above.

\begin{corollary}\label{cor:equivariant-busby}
Let $G$ be a second countable groupoid, $(I,\gamma)$ and $(A,\alpha)$ be two separable, $G$-$C^*$-algebras and let $\mu : A \to \mathcal{Q}(I)$ be a $G$-admissible, $*$-homomorphism. Then the following statements hold.
\begin{enumerate}
\item The $C_0(G^{(0)})$-algebra $B$ from Proposition \ref{prop:c0x-busby} can be uniquely made into a $G$-$C^*$-algebra such that $\iota$ and $p$ are $G$-equivariant.
\item The extension 
\begin{equation}\tag{$\eta$}
0 \rightarrow I \xrightarrow{\iota} B \xrightarrow{p} A \rightarrow 0
\end{equation}is a $G$-$C^*$-extension such that $\mu(\eta)=\mu$.
\end{enumerate}
\end{corollary}

\begin{proof}
We note that (2) follows from (1) and Proposition \ref{prop:c0x-busby}. To prove (1), we will use Lemma \ref{lem:continuity}. To that end, we note that for any $g \in G$ and $(m,a) \in B$, $(\gamma_g\xi_{s(g)}(m)\gamma_g^{-1},\alpha_g(a(s(g)))) \in B_{r(g)}$, by Proposition \ref{prop:c0x-busby} and Lemma \ref{lem:busby-equivariant}. The result now follows since $\mu$ is assumed to be $G$-admissible. Uniqueness follows from a careful adaptation of the trick in the proof of Proposition \ref{prop:c0x-busby}.
\end{proof}

\begin{corollary}
Let $G$ be a locally compact, Hausdorff and second countable group, $(A, \alpha)$ and $(I, \gamma)$ be two separable, $G$-$C^*$-algebras, and let $\mu \colon A \to \mathcal{Q}(I)$ be a $*$-homomorphism. Then $\mu$ is $G$-admissible if and only if $\mu$ is $G$-equivariant.
\end{corollary}

\begin{proof}
First, we assume that $\mu$ is $G$-admissible. Then by Corollary \ref{cor:equivariant-busby} and Lemma \ref{lem:busby-equivariant}, $\mu$ is $G$-equivariant. The converse follows from \cite{thomsen:equivariant-kk}*{Theorem 2.1} and Corollary \ref{cor:busby-g-admissible}.
\end{proof}

We remark that by Corollary \ref{cor:equivariant-busby} and Lemma \ref{lem:busby-equivariant}, it also follows that for a groupoid $G$, $G$-admissibility implies $G$-equivariance. Therefore it is natural to conjecture the following.

\begin{conjecture}
Let $G$ be a second countable groupoid, $(A, \alpha)$ and $(I, \gamma)$ be two separable, $G$-$C^*$-algebras, and let $\mu \colon A \to \mathcal{Q}(I)$ be a $*$-homomorphism. If $\mu$ is $G$-equivariant then $\mu$ is $G$-admissible.
\end{conjecture}

We end this section with a definition and an easy lemma.

\begin{definition}
Let $G$ be a groupoid, $(A,\alpha)$ be a $G$-$C^*$-algebra and $I$ be an ideal in $A$. We say that $I$ is a $G$-invariant ideal if $\alpha(s^*I)=r^*I$.
\end{definition}
  
\begin{lemma}
Let $(A,\alpha)$ be a $G$-$C^*$-algebra and $I$ be a $G$-invariant ideal in $A$. Then the quotient $A/I$ can be made into a $G$-$C^*$-algebra such that the extension 
\[
0 \rightarrow I \rightarrow A \rightarrow A/I \rightarrow 0  
\]is a $G$-$C^*$-extension.
\end{lemma}

\begin{proof}
It is easy to see that $A/I$ can be made into a $C_0(G^{(0)})$-algebra via $\theta_A$ such that the extension 
\[
0 \rightarrow I \rightarrow A \rightarrow A/I \rightarrow 0  
\]is a $C_0(G^{(0)})$-extension. The result now is immediate from the definition of a $G$-invariant ideal and the fact that the pullback functor is exact.
\end{proof}

\section{Some technical results on quasi-central approximate units}\label{sec:quasi-central} As the title suggests, in this section, we gather some technical results on quasi-central approximate units that form one of the main technical tools for our main theorem. More precisely, we present variations of the Kasparov lemma (\cite{kasparov:equivariant-kk}*{Lemma 1.4}) as presented in \cite{forough-gardella-thomsen:lifting}.

\begin{lemma}\label{lemma:quasi-central}
Let $J$ be a separable $C^*$-algebra, $0 \leq d \leq 1$ be a strictly positive element in $J$ and $M_0 \subseteq \mathcal{M}(J)$ be a separable $C^*$-subalgebra. Then there exists a countable approximate unit $(x_n)_{n \in \mathbb{N}}$ for $J$ contained in $C^*(d)$ with the following properties.
\begin{itemize}
\item For each $n \in \mathbb{N}$, $0 \leq x_n\leq 1$ and $x_{n+1}x_n=x_n$.
\item For each $b \in M_0$, $\left\lVert x_nb-bx_n \right\rVert \rightarrow 0$ as $n \rightarrow \infty$.
\end{itemize}
\end{lemma}

The following lemma is the groupoid analogue of \cite{kasparov:equivariant-kk}*{Lemma 1.4} proved in \cite{le_gall:equivariant-kk}*{Corollary 6.1} and \cite{tu:baum-connes-amenable}*{Corollary 4.4}.

\begin{lemma}\label{lemma:g-quasi-central}
Let $G$ be a second countable groupoid, $(J,\gamma)$ be a separable, $G$-$C^*$-algebra, $0 \leq d \leq 1$ be a strictly positive element in $J$ and $M_0 \subseteq \mathcal{M}(J)$ be a separable $C^*$-subalgebra. Then there exists a countable approximate unit $(u_n)_{n \in \mathbb{N}}$ of $J$ contained in $C^*(d)$ with the following properties. 
\begin{enumerate}
\item For each $n \in \mathbb{N}$, $0 \leq u_n\leq 1$ and $u_{n+1}u_n=u_n$.
\item For each $b \in M_0$, $\left\lVert u_nb-bu_n \right\rVert \rightarrow 0$ as $n \rightarrow \infty$.
\item For any compact subset $K \subseteq G$,
\[
  \max_{g \in K}\left\lVert \gamma_{g}(u_n(s(g))-u_n(r(g)) \right\rVert \rightarrow 0 \text{ as } n \rightarrow \infty.  
\]
\end{enumerate}
\end{lemma}

\begin{proof}
We apply Lemma \ref{lemma:quasi-central} to find a countable approximate unit $(x_n)_{n \in \N}$ contained in $C^*(d)$ satisfying (1)--(3). By the proof of \cite{tu:baum-connes-amenable}*{Lemma 4.3}, there exists a countable approximate unit $(u_n)_{n \in \mathbb{N}}$ such that for all $n \in \mathbb{N}$, $u_n \in \mathrm{Conv}(x_m)_{m \geq n}$. We observe that since each $u_n \in \mathrm{Conv}(x_m)_{m\geq n}$, $(u_n)_{n \in \mathbb{N}}$ satisfies (1)--(3) as well.
\end{proof}

We recall the following lemma. 

\begin{lemma}[{\cite{arveson:extensions}*{Lemma 1.1}}] \label{lemma:square-root}
Let $J$ be a $C^*$-algebra, $\epsilon >0$ and $f$ be a continuous function on $[0,1]$ satisfying $f(0)=0$. Then there is a $\delta >0$ such that for each pair $(a,e)$ of contractions in $J$ with $a \geq 0$, one has
\[
\left\lVert ae-ea \right\rVert \leq \delta \implies \left\lVert f(a)e-ef(a) \right\rVert \leq \epsilon.
\]
\end{lemma}

The following lemma is one of the main technical tools in the proof of our main theorem.

\begin{lemma}\label{lemma:averaging}
Let $G$ be a second countable groupoid, $(J,\gamma)$ be a separable, $G$-$C^*$-algebra, $0 \leq d \leq 1$ be a strictly positive element in $J$, $(W_n)_{n \in \mathbb{N}}$ be an increasing sequence of compact subsets in $\mathcal{M}(J)$ and $(K_n)_{n \in \mathbb{N}}$ be an increasing sequence of compact subsets in $G$. Then there exists a countable approximate unit $(u_n)_{n \in \N}$ of $J$ contained in $C^*(d)$ such that if we set
\[
\Delta_0=\sqrt{u_0}, \quad \Delta_n=\sqrt{u_n-u_{n-1}},
\]
for $n \geq 1$, the following properties are satisfied.
\begin{enumerate}
\item For each $n \in \N$, $0 \leq u_n \leq 1$ and $u_n=u_{n+1}u_n$.
\item For each $n \in \mathbb{N}$, $\Delta_{n+1}u_n=0$.
\item For any compact subset $K \subseteq G$, \[\max_{g \in K} \left\lVert \gamma_g(u_n(s(g))-u_n(r(g)) \right\rVert \rightarrow 0 \text{ as } n \rightarrow \infty.\]
\item For each $n \in \mathbb{N}$ and for any $m \in W_n$,
\[
  \left\lVert \Delta_n m-m\Delta_n \right\rVert \leq 1/n^2.  
\]
\item For each $n \in \mathbb{N}$ and for any $g \in K_n$,
\[
  \left\lVert \gamma_g(\Delta_n(s(g)))-\Delta_n(r(g)) \right\rVert \leq 1/n^2.  
\]
\item For each $n \in \mathbb{N}$, for any $m \in W_n$,
\[
  \left\lVert m(1-u_n) \right\rVert \leq \left\lVert q_J(m) \right\rVert + 1/n^2.  
\]
\end{enumerate}
\end{lemma}

\begin{proof}
The proof is similar to the proof of \cite{forough-gardella-thomsen:lifting}*{Lemma 2.3}. Let $M_0 \subseteq \mathcal{M}(J)$ be the separable $C^*$-subalgebra generated by $\bigcup_{n \in \mathbb{N}} W_n$ and let $(y_n)_{n \in \mathbb{N}}$ be a countable approximate unit of $J$ contained in $C^*(d)$ satisfying (1)--(4) of Lemma \ref{lemma:g-quasi-central}. By approximating the square root function on $[0,1]$ by polynomials, we see that given $\epsilon > 0$, there is a $\delta >0$ such that 
\[
\left\lVert \gamma_g(a(s(g)))-a(r(g)) \right\rVert \leq \delta \implies \left\lVert \gamma_g\left(\sqrt{a(s(g))}\right)-\sqrt{a(r(g))} \right\rVert \leq \epsilon,
\]for any positive contraction $a \in J$ and $g \in G$. Together with Lemma \ref{lemma:square-root}, this implies that for each $n \in \mathbb{N}$, there is a $\delta_n > 0$ such that given a positive contraction $a \in J$,
\[
\left\lVert \gamma_g(a(s(g)))-a(r(g)) \right\rVert \leq \delta_n \implies \left\lVert \gamma_g\left(\sqrt{a(s(g))}\right)-\sqrt{a(r(g))} \right\rVert \leq \frac{1}{n^2},  
\]for any $g \in K_n$ and 
\[
\left\lVert am-ma \right\rVert \leq \delta_n \implies \left\lVert \sqrt{a}m-m\sqrt{a} \right\rVert \leq \frac{1}{n^2}, 
\]for any $m \in W_n$. Now using (3) of Lemma \ref{lemma:g-quasi-central}, we find a subsequence $(u_n)_{n \in \mathbb{N}}$ of $(y_n)_{n \in \mathbb{N}}$ such that given any $n \in \mathbb{N}$,
\[
\left\lVert \gamma_g((u_{n}-u_{n-1})(s(g)))-(u_{n}-u_{n-1})(r(g)) \right\rVert \leq \delta_n,
\]for any $g \in K_n$ and
\[
\left\lVert (u_{n}-u_{n-1})m-m(u_{n}-u_{n-1}) \right\rVert \leq \delta_n,
\]for any $m \in W_n$. Using the fact that 
\[
\left\lVert q_J(m) \right\rVert=\lim_{n \to \infty}\left\lVert m(1-y_n) \right\rVert  
\]
for all $m \in \mathcal{M}(J)$ and passing to a further subsequence, we obtain $(u_n)_{n \in \mathbb{N}}$ satisfying (1)--(6).
\end{proof}

\section{The main results}\label{sec:main-results} We now come, as the title suggests, to the main results of this article. The approach is standard, as is already taken in \cite{forough-gardella-thomsen:lifting} and \cite{gabe:lifting}; namely, we first consider the case when the codomain is a Calkin algebra and then use the Busby invariant to pass to the general case. We begin with the following definition.

\begin{definition}
Let $G$ be a groupoid, $(A,\alpha)$ and $(B,\beta)$ be two $G$-$C^*$-algebras, $\phi : A \rightarrow B$ be a completely positive and contractive map, $F \subset A$ be a finite subset of $A$ and $K \subset G$ be a compact subset of $G$. The \emph{equivariance modulus} of $\phi$ with respect to $\alpha$, $\beta$, $F$ and $K$, denoted $\left\lVert (\phi,\alpha,\beta) \right\rVert_{F,K}$ is defined to be the number
\begin{equation}
\left\lVert (\phi,\alpha,\beta) \right\rVert_{F,K}=\max_{a \in F}\sup_{g \in K}\left\lVert \beta_g (\phi_{s(g)}(a(s(g))-\phi_{r(g)}(\alpha_g(a(s(g)))) \right\rVert. 
\end{equation}
\end{definition}

We remark that the above definition still makes sense when $G$ is second countable and $B$ is $\mathcal{M}(I)$ or $\mathcal{Q}(I)$ for a separable, $G$-$C^*$-algebra $(I,\gamma)$, with $\bar \gamma$, respectively $\tilde \gamma$, playing the role of $\beta$ and we will continue to write $\left\lVert (\phi,\alpha,\bar \gamma) \right\rVert_{F,K}$, respectively $\left\lVert (\phi,\alpha,\tilde \gamma) \right\rVert_{F,K}$, for the equivariance modulus. Now we can state the technical lemma for our main result.

\begin{lemma}\label{lem:lift-calkin}
Let $G$ be a second countable groupoid, $(A,\alpha)$ and $(I,\gamma)$ be two $G$-$C^*$\nobreakdash-algebras with $A$ and $I$ separable. Let $\phi : A \to \mathcal{M}(I)$ and $\psi : A \to \mathcal{Q}(I)$ be two $C_0(G^{(0)})$\nobreakdash-linear, completely positive and contractive maps such that $q_I \circ \phi=\psi$. Then given a finite subset $F \subset A$, a compact set $K \subset G$ and $\epsilon >0$, there exists a $C_0(G^{(0)})$-linear, completely positive and contractive map $\phi' : A \to \mathcal{M}(I)$ such that $q_I \circ \phi'=\psi$ and
\begin{equation}\label{eq:equivariance-modulus}
\left\lVert (\phi',\alpha,\bar \gamma) \right\rVert_{F,K} \leq \left\lVert (\psi,\alpha,\tilde \gamma) \right\rVert_{F,K} + \epsilon.
\end{equation}
\end{lemma}

\begin{proof}
We begin the proof by fixing an increasing sequence of finite subsets $(F_n)_{n \in \mathbb{N}}$ of $A$ with dense union in $A$ and $F \subseteq F_0$. Similarly, we fix an increasing sequence of compact subsets $(K_n)_{n \in \mathbb{N}}$ of $G$ whose union equals $G$ itself and $K \subseteq K_0$. Now we set $W_n=\phi(F_n)$, $n \in \mathbb{N}$, and by Lemma \ref{lemma:averaging} we obtain an approximate unit $(u_n)_{n \in \N}$ of $I$, hence the sequence $(\Delta_n)_{n \in \mathbb{N}}$ satisfying the conditions (1)--(6) as in Lemma \ref{lemma:averaging} for the sequences $(W_n)_{n \in \N}$ and $(K_n)_{n \in \N}$. By \cite{manuilov-thomsen:e-theory-special-case}*{Lemma 3.1}, we see that for every $n \in \mathbb{N}$, the sequence
\begin{equation}
\sum_{l=n}^k \Delta_l\phi(a)\Delta_l  
\end{equation}converges in the strict topology of $\mathcal{M}(I)$ as $k \to \infty$ and
\begin{equation}
\left\lVert \sum_{l=n}^\infty \Delta_l\phi(a)\Delta_l \right\rVert \leq 1.  
\end{equation}Therefore, for each $n \in \mathbb{N}$, we can define a completely positive and contractive map
\begin{equation}
\phi_n \colon A \to \mathcal{M}(I),    
\end{equation}by setting for each $a \in A$,
\begin{equation}
\phi_n(a)=\sum_{l=n}^\infty \Delta_l \phi(a) \Delta_l.  
\end{equation}We note that, since $\phi$ is $C_0(G^{(0)})$-linear, so is $\phi_n$, for each $n \in \mathbb{N}$. Furthermore, property (4) in Lemma \ref{lemma:averaging} implies that
\begin{equation}
\sum_{l=n}^{k} \left\lVert \phi(a) \Delta_l^2 - \Delta_l \phi(a) \Delta_l \right\rVert \leq \sum_{l=n}^{k} \left\lVert \phi(a) \Delta_l - \Delta_l \phi(a) \right\rVert \leq \sum_{l=n}^{k} \frac{1}{l^{2}},
\end{equation}for all $n \in \mathbb{N}$, $a \in F_n$ and $k > n$. Therefore, the series 
\begin{equation}
\sum_{l=0}^{\infty} \left(\phi(a) \Delta_l^2 -\Delta_l \phi(a) \Delta_l  \right),
\end{equation}converges in norm, for all $a \in \bigcup_{l \in \N}F_l$, hence for all $a \in A$, to an element in $I$ as each summand is an element of $I$. Since $\sum_{l=0}^{\infty}\Delta_l^2$ converges strictly to $1$, we have
\begin{equation}
\phi(a) - \phi_n(a) = \sum_{l=0}^{\infty} (\phi(a) \Delta_l^2- \Delta_l \phi(a) \Delta_l) + \sum_{l=0}^{n-1} \Delta_l \phi(a) \Delta_l,
\end{equation}for all $a\in A$ and $n \in \mathbb{N}$. Since each summand of the above sum is in $I$, $\phi(a) - \phi_n(a)$ is in $I$ as well, i.e., $q_I \circ \phi_n=q_I \circ \phi=\psi$, for all $n \in \N$. Now since $((u_n)(x))_{n \in \N}$ is an approximate unit of $I(x)$ for all $x \in G^{(0)}$, 
\begin{equation}
\left\lVert q_{I(x)}(m) \right\rVert=\lim_{n \rightarrow \infty}\left\lVert m((1-u_n)(x)) \right\rVert  
\end{equation}for $m \in \mathcal{M}(I(x))$. We can then find $n_0 \in \mathbb{N}$ such that whenever $a \in F$, $g \in K$ and $n \geq n_0$,
\begin{subequations}
\begin{equation}\label{eq:n0}
\sum_{l=n}^\infty \frac{1}{l^2} \leq \frac{\epsilon}{4\left\lVert a \right\rVert},
\end{equation}\text{and}
\begin{equation}\label{eq:quotient-norm}
\begin{aligned}
&{}\left\lVert (\bar \gamma_g (\phi_{s(g)}(a(s(g))))-\phi_{r(g)}(\alpha_g(a(s(g))))(1-u_n)(r(g)) \right\rVert\\
\leq &{} \left\lVert q_{I(r(g))}(\bar \gamma_g (\phi_{s(g)}(a(s(g))))-\phi_{r(g)}(\alpha_g(a(s(g)))) \right\rVert + \frac{\epsilon}{2}\\
=&{}\left\lVert \tilde \gamma_g(\psi_{s(g)}(a(s(g)))-\psi_{r(g)}(\alpha_g(a(s(g))))) \right\rVert + \frac{\epsilon}{2},
\end{aligned}      
\end{equation}
\end{subequations}where in the last step we use the commutativity of the square \eqref{eq:alphabar-alphatilde}. By (5) of Lemma \ref{lemma:averaging}, we obtain for all $l \in \mathbb{N}$ and for all $a \in F$,
\begin{equation}\label{eq:average}
\begin{aligned}
&{} \left\lVert \gamma_g(\Delta_l(s(g)))\bar \gamma_g(\phi_{s(g)}(a(s(g)))\gamma_g(\Delta_l(s(g)))-\Delta_l(r(g))\bar \gamma_g(\phi_{s(g)}(a(s(g)))\Delta_l(r(g)) \right\rVert \\ 
&{} \leq \frac{2\left\lVert a \right\rVert}{l^2}.
\end{aligned}
\end{equation}Therefore, for all $n \geq n_0$, $a \in F$ and $g \in K$, we have
\begin{equation}\label{eq:estimate-1}
\begin{aligned}
&{} \left\lVert \bar \gamma_g((\phi_n)_{s(g)}(a(s(g)))-\sum_{l=n}^\infty \Delta_l(r(g))\bar \gamma_g(\phi_{s(g)}(a(s(g)))\Delta_l(r(g)) \right\rVert \\
=&{} \left\lVert \sum_{l=n}^\infty \gamma_g(\Delta_l(s(g))) \bar \gamma_g(\phi_{s(g)}(a(s(g)))\gamma_g(\Delta_l(s(g)))-\sum_{l=n}^\infty \Delta_l(r(g))\bar \gamma_g(\phi_{s(g)}(a(s(g)))\Delta_l(r(g)) \right\rVert\\
\leq &{} \sum_{l=n}^\infty \frac{2\left\lVert a \right\rVert}{l^2} \leq \frac{\epsilon}{2},     
\end{aligned}    
\end{equation}where the first inequality is by \eqref{eq:average} and the second inequality is by \eqref{eq:n0}. We also note that we have used the identity ($a \in A$), \[(\phi_n)_x(a(x))=\sum_{l=n}^\infty \Delta_l(x) \phi_x(a(x))\Delta_l(x)\]which holds since $\xi_x$ is continuous with respect to the strict topology, for all $x \in G^{(0)}$; and the fact that $\bar \gamma_g$ is continuous with respect to the strict topology, for all $g \in G$. We now show that $\phi_n$ satisfies the inequality \eqref{eq:equivariance-modulus}in the statement for $n > n_0$. So we fix $N > n_0$ and observe that for $n \geq N$, $\Delta_n=(1-u_{n_0})\Delta_n$, since $u_{n_0}\Delta_n=0$. Therefore, for $l > k \geq N$, using Cauchy-Schwarz inequality for the Hilbert module $\mathcal{M}(I(r(g)))^{l-k+1}$ at the first step, and \eqref{eq:quotient-norm} at the second, we obtain
\begin{equation}
\begin{aligned}
&{} \left\lVert \sum_{n=k}^l\Delta_n(r(g))(\bar \gamma_g\phi_{s(g)}(a(s(g)))-\phi_{r(g)}(\alpha_g(a(s(g))))\Delta_n(r(g)) e(r(g))\right\rVert \\
= &{} \left\lVert \sum_{n=k}^l\Delta_n(r(g))(\bar \gamma_g(\phi_{s(g)}(a(s(g))))-\phi_{r(g)}(\alpha_g(a(s(g))))(1-u_{n_0})(r(g))\Delta_n(r(g))e(r(g)) \right\rVert \\
\leq &{} \left\lVert (\bar \gamma_g(\phi_{s(g)}(a(s(g))))- \phi_{r(g)}(\alpha_g(a(s(g))))(1-u_{n_0})(r(g)) \right\rVert \left\lVert \sum_{n=k}^le(r(g))^*\Delta_n^2(r(g))e(r(g)) \right\rVert \\
\leq &{} \left(\left\lVert \tilde \gamma_g(\psi_{s(g)}(a(s(g)))-\psi_{r(g)}(\alpha_g(a(s(g)))) \right\rVert + \frac{\epsilon}{2} \right) \left\lVert (e^*(u_l-u_k)e)(r(g)) \right\rVert
\end{aligned}  
\end{equation}for all $a \in F$, $g \in K$ and $e \in I$. We note that $e^*(u_l-u_k)e$ converges to zero in norm as $l,k \to \infty$, since $(u_l)_{l \in \N}$ is an approximate unit of $I$. Hence, 
\begin{equation}
\sum_{n=N}^\infty \Delta_n(r(g))(\bar \gamma_g\phi_{s(g)}(a(s(g)))-\phi_{r(g)}(\alpha_g(a(s(g))))(r(g))\Delta_n(r(g))
\end{equation}converges in the strict topology of $\mathcal{M}(I(r(g)))$ such that
\begin{equation}\label{eq:estimate-3}
\begin{aligned}
&{} \left\lVert \sum_{n=N}^\infty \Delta_n(r(g)) (\bar \gamma_g\phi_{s(g)}(a(s(g)))-\phi_{r(g)}(\alpha_g(a(s(g))))(r(g))\Delta_n(r(g)) \right\rVert \\
\leq &{} \left\lVert \tilde \gamma_g(\psi_{s(g)}(a(s(g)))-\psi_{r(g)}(\alpha_g(a(s(g)))) \right\rVert + \frac{\epsilon}{2},  
\end{aligned}
\end{equation}for all $a \in F$ and $g \in K$. Now we have all the estimates to end the proof, so to that end, fix $a \in F$ and $g \in K$. Then
\begin{equation}
\begin{aligned}
&{} \left\lVert (\bar \gamma_g((\phi_N)_{s(g)}(a(s(g))))-(\phi_N)_{r(g)}\alpha_g(a(s(g)))) \right\rVert \\
\leq &{} \left\lVert (\bar \gamma_g((\phi_N)_{s(g)}(a(s(g))))-\sum_{n=N}^\infty \Delta_n(r(g))\bar \gamma_g(\phi_{s(g)}(a(s(g))))\Delta_n(r(g))) \right\rVert \\
+ &{} \left\lVert (\sum_{n=N}^\infty \Delta_n(r(g))\bar \gamma_g(\phi_{s(g)}(a(s(g))))\Delta_n(r(g))-(\phi_N)_{r(g)}\alpha_g(a(s(g)))) \right\rVert \\
\leq &{} \frac{\epsilon}{2} + \left\lVert (\sum_{n=N}^\infty \Delta_n(r(g))(\bar \gamma_g\phi_{s(g)}(a(s(g)))-\phi_{r(g)}(\alpha_g(a(s(g))))\Delta_n(r(g))) \right\rVert \\
\leq &{} \epsilon + \left\lVert \tilde \gamma_g(\psi_{s(g)}(a(s(g)))-\psi_{r(g)}(\alpha_g(a(s(g)))) \right\rVert,
\end{aligned} 
\end{equation}where, the second inequality is by \eqref{eq:estimate-1} and the definition of $\phi_N$; and the final inequality is by \eqref{eq:estimate-3}. This completes the proof.
\end{proof}

Now we are in the position to prove our main result.

\begin{theorem}\label{theo:lift}
Let $G$ be a second countable groupoid and $(A,\alpha)$ be a separable, $G$-$C^*$-algebra. Let  
\[
0 \rightarrow (I,\gamma) \rightarrow (B,\beta) \xrightarrow{p} (D,\delta) \rightarrow 0
\]be a $G$-$C^*$-extension with $B$ separable. Let $\phi : A \rightarrow B$ and $\psi : A \rightarrow D$ be two $C_0(G^{(0)})$-linear, completely positive and contractive maps such that $p \circ \phi=\psi$. Then given $\epsilon > 0$, a finite subset $F \subset A$ of $A$, a compact subset $K \subset G$ of $G$, there exists a $C_0(G^{(0)})$-linear, completely positive and contractive map $\phi' : A \rightarrow B$ such that $p \circ \phi'=\psi$ and 
\begin{equation}\label{eq:main-equivariance-modulus}
\left\lVert (\phi',\alpha,\beta) \right\rVert_{F,K} \leq \left\lVert (\psi,\alpha,\delta) \right\rVert_{F,K} + \epsilon.
\end{equation}
\end{theorem}

\begin{proof}
We begin by considering the $G$-$C^*$-extension 
\begin{equation}\tag{$\eta$}
0 \rightarrow I \rightarrow B \rightarrow D \rightarrow 0,
\end{equation}and the associated $*$-homomorphims 
\begin{equation}
\bar \mu(\eta) : B \rightarrow \mathcal{M}(I), \quad \mu(\eta) : D \rightarrow \mathcal{Q}(I).  
\end{equation}
Setting 
\begin{equation}
E=\{(m,c) \in \mathcal{M}(I) \oplus D : q_I(m)=\mu(\eta)(c)\},
\end{equation}
and
\begin{equation}
\Theta : B \to E, \quad \Theta(b)=(\bar \mu(\eta)(b),p(b)) \text{ for all } b \in B,
\end{equation}by Corollary \ref{cor:equivariant-busby}, we see that $E$ can be made into a $G$-$C^*$-algebra and $\Theta$ is a $G$-equivariant isomorphism.
Let
\begin{equation}\label{eq:descent}
\dot{\phi}=\bar \mu(\eta) \circ \phi : A \rightarrow \mathcal{M}(I), \text{ and } \dot{\psi}=\mu(\eta) \circ \psi : A \rightarrow \mathcal{Q}(I),
\end{equation}so that
\begin{equation}\label{eq:descent-lift}
q_I \circ \dot{\phi}=q_I \circ \bar \mu(\eta) \circ \phi=\mu(\eta) \circ p \circ \phi=\mu(\eta) \circ \psi=\dot{\psi},  
\end{equation}
i.e., we are in the setting of Lemma \ref{lem:lift-calkin}. Therefore we find
\begin{equation}
\dot{\phi}' : A \rightarrow \mathcal{M}(I) \text{ such that } q_I \circ \dot{\phi}'=\dot{\psi}  
\end{equation}
and 
\begin{equation}\label{eq:descent-control}
\left\lVert (\dot \phi',\alpha,\bar \gamma) \right\rVert_{F,K} \leq \left\lVert (\dot \psi,\alpha,\tilde \gamma) \right\rVert_{F,K} + \epsilon.    
\end{equation}
This implies that for all $a \in A$, $q_I(\dot \phi'(a))=\dot \psi(a)=\mu(\eta)(\psi(a))$, i.e., $(\dot \phi'(a), \psi(a)) \in E$. We can therefore define 
\begin{equation}
\phi' : A \to B \quad \phi'(a)=\Theta^{-1}(\dot \phi'(a),\psi(a)), \quad a \in A,
\end{equation} 
which is then clearly a $C_0(G^{(0)})$-linear, completely positive and contractive map such that $p \circ \phi'=\psi$. To end the proof, we need the control on the equivariance modulus of $\phi'$. To that end, we fix $g \in G$ and $a \in A$ and observe that 
\begin{equation}
\begin{aligned}
\tilde \gamma_g(\dot \psi_{s(g)}(a(s(g))))=&{} \tilde \gamma_g(\mu(\eta)_{s(g)}(\psi_{s(g)}(a(s(g)))))\\ 
=&{} \tilde \gamma_g(q_{I_{s(g)}}(\bar \mu(\eta)_{s(g)}(\phi_{s(g)}(a(s(g))))))\\
=&{} q_{I_{r(g)}}(\bar \gamma_g(\bar \mu(\eta)_{s(g)}(\phi_{s(g)}(a(s(g))))))\\
=&{} q_{I_{r(g)}}(\bar \mu(\eta)_{r(g)}(\beta_g(\phi_{s(g)}(a(s(g))))))\\
=&{} \mu(\eta)_{r(g)}(p_{r(g)}(\beta_g(\phi_{s(g)}(a(s(g))))))\\
=&{} \mu(\eta)_{r(g)}(\delta_g((p_{s(g)}(\phi_{s(g)}(a(s(g)))))))\\
=&{} \mu(\eta)_{r(g)}(\delta_g(\psi_{s(g)}(a(s(g))))),
\end{aligned}
\end{equation}
where the first equality is by definition of $\dot \psi$, \eqref{eq:descent}; the second equality is by \eqref{eq:descent-lift}; the third equality is by \eqref{eq:alphabar-alphatilde}; the fourth equality is by Lemma \ref{lem:mubar-equivariant}; the fifth equality is by definition of $\mu(\eta)$; the sixth equality is because $p$ is $G$-equivariant and the final equality is by hypothesis. This together with \eqref{eq:descent-control} imply
\begin{equation}
\begin{aligned}
\left\lVert (\dot \phi',\alpha,\bar \gamma) \right\rVert_{F,K} \leq \left\lVert (\psi,\alpha,\delta) \right\rVert_{F,K} + \epsilon.
\end{aligned}  
\end{equation}which, combined with the fact that $\Theta$ is $G$-equivariant, in turn yield \eqref{eq:main-equivariance-modulus}, as desired. This completes the proof.
\end{proof}

\begin{remark}\label{rem:existence-lift}
We remark that the existence of $\phi$ in the theorem is a nontrivial question. If $A$ is $C_0(X)$-nuclear and $D$ is separable, then Proposition \ref{prop:c0x-choi-effros} guarantees the existence of $\phi$. We expound on this issue in more detail in the Appendix \ref{sec:appendix} below.
\end{remark}

Since $A$ is separable and $G$ is second countable, we obtain the following theorem as a corollary to Theorem \ref{theo:lift}.

\begin{theorem}
Let $G$ be a second countable groupoid and $(A,\alpha)$ be a separable, $G$-$C^*$-algebra. Let  
\[
0 \rightarrow (I,\gamma) \rightarrow (B,\beta) \xrightarrow{p} (D,\delta) \rightarrow 0
\]be a $G$-$C^*$-extension with $B$ separable. Let $\phi : A \rightarrow B$ and $\psi : A \rightarrow D$ be two $C_0(G^{(0)})$-linear, completely positive and contractive maps with $\psi$ equivariant and such that $p \circ \phi=\psi$. Then there exists a sequence of $C_0(G^{(0)})$-linear, completely positive and contractive maps $\phi'_n : A \rightarrow B$, $n \in \mathbb{N}$, such that $p \circ \phi'_n=\psi$ and 
\begin{equation}
\lim_{n \rightarrow \infty} \sup_{g \in K}\left\lVert \beta_g ((\phi'_n)_{s(g)}(a(s(g))-(\phi'_n)_{r(g)}(\alpha_g(a(s(g)))) \right\rVert=0,
\end{equation}for all $a \in A$ and for all compact subsets $K \subseteq G$.
\end{theorem}

We thank Valerio Proietti for pointing out that the above theorem is similar to \cite{higson-kasparov:e-theory}*{Theorem 8.1}; however, the precise connection is yet to be established.

\appendix
\section{}\label{sec:appendix} The purpose of this appendix is to clarify some of the lifting results related to this article from \cites{kasparov-skandalis:group-actions-buildings,bauval:nuclear-rkk,gabe:lifting} and to address the issue mentioned in Remark \ref{rem:existence-lift} above. We begin with the latter.

\begin{definition}[{\cite{kasparov-skandalis:group-actions-buildings}*{Section 6.2}}, \cite{bauval:nuclear-rkk}*{Definition 2.1}]
Let $A$ and $B$ be two $C_0(X)$-algebras.
\begin{itemize}
  \item A completely positive, $C_0(X)$-linear map $\phi : A \to B$ is said to be $C_0(X)$-\emph{factorable} if it admits a factorization $\phi=\sigma \circ \tau$, where $n \in \mathbb{N}$, $\tau : A \to M_n(\mathbb{C}) \otimes C_0(X)$ and $\sigma : M_n(\mathbb{C}) \otimes C_0(X) \to B$ are completely positive, $C_0(X)$-linear maps.
  \item $\phi$ is said to be strongly $C_0(X)$-factorable if moreover $\sigma$ and $\tau$ can be chosen to be contractive.
\end{itemize}
\end{definition}

Let $\mathrm{fact}^1(X;A,B)$ denote the set of all $C_0(X)$-factorable and contractive maps; further, let $\mathrm{fact}^s(X;A,B)$ denote the subset of all strongly $C_0(X)$-factorable, completely positive maps. By \cite{bauval:nuclear-rkk}*{Proposition 2.3}, $\mathrm{fact}^s(X;A,B)$ is dense in $\mathrm{fact}^1(X;A,B)$ in the point-norm topology. When $X=\{pt\}$, the two coincide. Let us also write $\mathrm{cp}^1(X;A,B)$ for the set of all $C_0(X)$-linear, completely positive and contractive maps.

\begin{definition}
Let $A$ and $B$ be two $C_0(X)$-algebras. 
\begin{itemize}
  \item $\phi \in \mathrm{cp}^1(X;A,B)$ is said to be $C_0(X)$-\emph{nuclear} if it is in the closure of $\mathrm{fact}^1(X;A,B)$ in the point-norm topology.
  \item $A$ is said to be $C_0(X)$-\emph{nuclear} if the identity map $\mathrm{id}_A : A \rightarrow A$ is $C_0(X)$-nuclear.
\end{itemize}
\end{definition}

We note that $A$ being $C_0(X)$-nuclear in particular implies that $A$ is nuclear in the ordinary sense. Thus for each $x \in X$, $A(x)$ is nuclear as well. The following (restatement of a) striking theorem of Bauval provides a converse.

\begin{theorem}[{\cite{bauval:nuclear-rkk}*{Theorem 7.2}}]
Let $A$ be a $C_0(X)$-algebra. Then the following are equivalent.
\begin{itemize}
  \item $A$ is $C_0(X)$-nuclear.
  \item $A$ is nuclear and is a continuous field of $C^*$-algebras.
\end{itemize}  
\end{theorem}

In \cite{kasparov-skandalis:group-actions-buildings}, the authors obtain the following analogue of the Choi-Effros lifting theorem, more generally for $C_0(X)$-nuclear maps. 

\begin{proposition}[{\cite{kasparov-skandalis:group-actions-buildings}*{Section 6.2}}] \label{prop:c0x-choi-effros} 
Let $A$ be a separable, $C_0(X)$-nuclear $C^*$-algebra; let
\[
0 \rightarrow I \rightarrow B \xrightarrow{p} D \rightarrow 0  
\]be a $C_0(X)$-extension with $B$ separable and let $\psi \in \mathrm{cp}^1(X;A,D)$. Then there is a $\phi \in \mathrm{cp}^1(X;A,B)$ such that $p \circ \phi=\psi$.
\end{proposition}

We note that this proposition ensures the existence of $\phi$ as in Remark \ref{rem:existence-lift}. Now, although the proof is a slight adaptation of Arveson's proof in \cite{arveson:extensions}, we sketch it nevertheless, for the sake of completeness. To that end, we fix a separable, $C_0(X)$-nuclear $C^*$-algebra $A$ and a $C_0(X)$-extension
\[
0 \rightarrow I \rightarrow B \xrightarrow{p} D \rightarrow 0,
\]with $B$ separable. We write
\[
\mathrm{lift}(X;A,D)=\{\psi \in \mathrm{cp}^1(X;A,D) : \text{there is a } \phi \in \mathrm{cp}^1(X;A,B) \text{ such that } p \circ \phi=\psi \}.  
\]Then the following lemma is the $C_0(X)$-version of \cite{arveson:extensions}*{Theorem 6} whose proof is identical and hence omitted.

\begin{lemma}\label{lem:liftable-closed}
$\mathrm{lift}(X;A,D)$ is closed in $\mathrm{cp}^1(X;A,D)$ in the point-norm topology.
\end{lemma}

The next step is to show that ``factorable maps are liftable''.

\begin{lemma}\label{lem:factorable-lifts}
$\mathrm{fact}^1(X;A,D) \subseteq \mathrm{lift}(X;A,D)$. 
\end{lemma}

The above lemma follows immediately from the following one.

\begin{lemma}\label{lem:matrix-lifts}
$\mathrm{cp}^1(X;M_n(\mathbb{C}) \otimes C_0(X),D) \subseteq \mathrm{lift}(X;M_n(\mathbb{C}) \otimes C_0(X),D)$.
\end{lemma}

\begin{proof}
To begin with, let us write $B_b=\{m \in \mathcal{M}(B) : fm \in B \text{ for all } f \in C_0(X)\}$ as in \cite{le_gall:equivariant-kk} (which is denoted $\mathcal{M}^{\#}(B)$ in \cite{fieux:pi-algebra}). Let $\bar p : \mathcal{M}(B) \to \mathcal{M}(D)$ be the canonical, surjective extension of $p : B \rightarrow D$, see for instance \cite{akemann-pedersen-tomiyama:multipliers}*{Theorem 4.2}. Since $p$ is a $C_0(X)$-linear, $\bar p$ maps $B_b$ onto $D_b$. Now let $\psi \in \mathrm{cp}^1(X;M_n(\mathbb{C}) \otimes C_0(X),D)$ and define $\bar \psi : M_n(\mathbb{C}) \to D_b$ by 
\[
\bar \psi(a)=\lim_{\lambda} \psi(a \otimes f_{\lambda}),  
\]where $(f_{\lambda})$ is an approximate unit for $C_0(X)$ with $0 \leq f_{\lambda} \leq 1$, see the proof of \cite{fieux:pi-algebra}*{Proposition 3.6}. It is easy to see that $\bar \psi$ is completely positive and contractive and hence there exists a completely positive and contractive $\bar \phi : M_n(\mathbb{C}) \to B_b$ such that $\bar p \circ \bar \phi=\bar \psi$. Now set $\phi : M_n(\mathbb{C}) \otimes C_0(X) \to B$ to be
\[
\phi(a \otimes f)=f\bar \phi(a).   
\]It is clear that $\phi \in \mathrm{cp}^1(X;M_n(\mathbb{C}) \otimes C_0(X),B)$ such that $p \circ \phi=\psi$.
\end{proof}
 
\begin{proof}[{Proof of Proposition \ref{prop:c0x-choi-effros}}]
The result follows from Lemmas \ref{lem:liftable-closed}, \ref{lem:factorable-lifts} and \ref{lem:matrix-lifts}.
\end{proof}

Let us remark that in \cite{bauval:nuclear-rkk}*{Section 3}, the author obtains similar results which coincide with the ones presented above in the case when $B$ is unital. This completes the second purpose of this appendix and we now move onto the first. In \cite{gabe:lifting}, the author proves a sweeping generalization of Proposition \ref{prop:c0x-choi-effros} in the context of $C^*$-algebras over topological spaces. One should note, however, that \cite{gabe:lifting} does not refer to \cite{kasparov-skandalis:group-actions-buildings} or \cite{bauval:nuclear-rkk} and it is therefore reasonable to relate the results of the former with the latter ones. Let us begin with a few notations. Let $\mathfrak{X}$ be a topological space and $A$ be a $C^*$-algebra. We will write $\mathbb{O}(\mathfrak{X})$ for the complete lattice of open sets of $\mathfrak{X}$ and $\mathbb{I}(A)$ for the complete lattice of closed, two-sided ideals of $A$.

\begin{definition}
Let $\mathfrak{X}$ be a topological space. An $\mathfrak{X}$-$C^*$-algebra is a pair $(A,\Theta_A)$, where $A$ is a $C^*$-algebra and $\Theta_A : \mathbb{O}(\mathfrak{X}) \rightarrow \mathbb{I}(A)$ is an order preserving map, i.e., a map such that if $U \subseteq V$ in $\mathbb{O}(\mathfrak{X})$ then $\Theta_A(U) \subseteq \Theta_A(V)$.
\end{definition}

It is customary to write only $A$ for an $\mathfrak{X}$-$C^*$-algebra $(A,\Theta_A)$, suppressing the structure map $\Theta_A$. It is also customary to write $A(U)$ for $\Theta_A(U)$, whenever $U \in \mathbb{O}(\mathfrak{X})$. 

A remarkable feature (at least to the authors!) of the general nature of $\mathfrak{X}$-$C^*$-algebras is the following. Let $A$ be a $\mathfrak{X}$-$C^*$-algebra. Then $\mathcal{M}(A)$ and $\mathcal{Q}(A)$ can both be made into $\mathfrak{X}$-$C^*$-algebras as follows. Given $U \in \mathbb{O}(\mathfrak{X})$, one sets 
\begin{subequations}
\begin{equation}
  \mathcal{M}(A)(U)=\{m \in \mathcal{M}(A) : m(a) \in A(U) \text{ for all } a \in A\},
\end{equation}
\begin{equation}
 \mathcal{Q}(A)(U)=q_A(\mathcal{M}(A)(U)). 
\end{equation}
\end{subequations}

\begin{definition}
Let $A$ and $B$ be two $\mathfrak{X}$-$C^*$-algebras. A $*$-homomorphism (or a completely positive map) $\phi : A \rightarrow B$ is said to be $\mathfrak{X}$-equivariant if $\phi(A(U)) \subseteq B(U)$ for all $U \in \mathbb{O}(\mathfrak{X})$.
\end{definition}

Because of the general nature of $\mathfrak{X}$-$C^*$-algebras, one oftentimes needs some continuity conditions, made precise in the following definition.

\begin{definition}
Let $A$ be an $\mathfrak{X}$-$C^*$-algebra. $A$ is said to be
\begin{itemize}
  \item \emph{finitely lower semi-continuous} if $A(\mathfrak{X})=A$ and given $U,V \in \mathbb{O}(\mathfrak{X})$, one has
  \[
  A(U) \cap A(V)=A(U \cap V);  
  \]
  \item \emph{lower semi-continuous} if $A(\mathfrak{X})=A$ and given a family $(U_{\lambda}) \subseteq \mathbb{O}(\mathfrak{X})$, one has 
  \[
  \bigcap_{\lambda}A(U_{\lambda})=A(U),  
  \]where $U$ is the interior of $\bigcap_{\lambda} U_{\lambda}$;
  \item \emph{finitely upper semi-continuous} if $A(\emptyset)=0$ and given $U,V \in \mathbb{O}(\mathfrak{X})$, one has 
  \[
  A(U)+A(V)=A(U \cup V);  
  \]
  \item \emph{monotone upper semi-continuous} if given an increasing net $(U_{\lambda}) \subseteq \mathbb{O}(\mathfrak{X})$, one has 
  \[
  \overline{\bigcup_{\lambda}A(U_{\lambda})}=A\left(\bigcup_{\lambda}U_{\lambda}\right);  
  \]
  \item \emph{upper semi-continuous} if it is finitely and monotone upper semi-continuous.
\end{itemize}
\end{definition}

Let $\mathfrak{X}$ be a locally compact, Hausdorff space $X$ (note the difference in notation as we exclusively used $X$ for a \emph{space}). By \cite{meyer-nest:bootstrap-class}*{Section 2}, there is a bijective correspondence between \emph{finitely lower semi-continuous} and \emph{upper semi-continuous} $\mathfrak{X}$-$C^*$-algebras and $C_0(X)$-algebras. Furthermore, if $A$ and $B$ are two $C_0(X)$-algebras and $\phi : A \rightarrow B$ is $*$-homomorphism or a completely positive map, then the following are equivalent.
\begin{itemize}
  \item $\phi$ is $C_0(X)$-linear.
  \item $\phi$ is $\mathfrak{X}$-equivariant.
  \item $\phi([\mathfrak{m}_xA]) \subseteq [\mathfrak{m}_xB]$ for all $x \in X$, i.e., $\phi$ descends to $\phi_x : A(x) \rightarrow B(x)$.
\end{itemize}Actually, more is true.

\begin{lemma}\label{lem:x-multiplier}
Let $\mathfrak{X}$ be locally compact, Hausdorff space $X$, $A$ and $B$ be two $C_0(X)$-algebras, $B$ be separable and $\phi : A \to \mathcal{M}(B)$ be a $*$-homomorphism (or a completely positive map). Then the following are equivalent.
\begin{enumerate}
  \item $\phi$ is $C_0(X)$-linear.
  \item $\phi$ is $\mathfrak{X}$-equivariant.
  \item For each $x \in X$, $\phi$ induces a $*$-homomorphism (or a completely positive map) $\phi_x : A(x) \to \mathcal{M}(B(x))$ such that $\phi_x(a(x))=\xi_x(\phi(a))$.
\end{enumerate}
\end{lemma}

\begin{proof}
(1)$\implies$(2). We recall that for $U \in \mathbb{O}(\mathfrak{X})=\mathbb{O}(X)$, $A(U)=C_0(U)A$. Then
\[
\phi(A(U))(B)=\phi(C_0(U)A)(B)=C_0(U)\phi(A)(B) \subseteq C_0(U)B=B(U),
\]i.e., $\phi(A(U)) \subseteq \mathcal{M}(B)(U)$ for all $U \in \mathbb{O}(\mathfrak{X})=\mathbb{O}(X)$.

\medskip

\noindent (2)$\implies$(3). It is enough to observe that for every $x \in X$, $\phi(a) \in \mathcal{M}(B)(X \setminus \{x\})$, whenever $a \in \mathfrak{m}_xA=A(X \setminus \{x\})$.

\medskip

\noindent (3)$\implies$(1). We fix $f \in C_0(X)$ and $a \in A$. Then for every $x \in X$,
\[
\xi_x(\phi(fa)-\theta_B(f)\phi(a))=\phi_x(f(x)a(x))-f(x)\phi_x(a(x))=0,
\]i.e., $((\phi(fa)-\theta_B(f)\phi(a))(b))(x)=0$ for all $x \in X$ and all $b \in B$, i.e., $\phi(fa)-\theta_B(f)\phi(a)=0$, i.e., $\phi$ is $C_0(X)$-linear.
\end{proof}

\begin{lemma}\label{lem:x-calkin}
Let $\mathfrak{X}$ be a locally compact, Hausdorff space $X$, $A$ and $B$ be two $C_0(X)$-algebras, and $\psi : A \to \mathcal{Q}(B)$ be a $C_0(X)$-linear $*$-homomorphism (or a completely positive map) such that there is a $C_0(X)$-linear, completely positive lift $\phi : A \rightarrow \mathcal{M}(B)$ with $q_B \circ \phi=\psi$. Then $\psi$ is $\mathfrak{X}$-equivariant.
\end{lemma}

\begin{proof}
We fix $U \in \mathbb{O}(\mathfrak{X})=\mathbb{O}(X)$. Then
\[
\psi(A(U))=q_B(\phi(A(U))) \subseteq q_B(\mathcal{M}(B)(U))=\mathcal{Q}(B)(U),  
\]which completes the proof.
\end{proof}

After these preliminaries, we can state (some special cases of) the results from \cite{gabe:lifting}.

\begin{proposition}[{\cite{gabe:lifting}*{Proposition 5.4}}]
Let $\mathfrak{X}$ be a topological space, $A$ be a separable, nuclear, lower semi-continuous $\mathfrak{X}$-$C^*$-algebra, $B$ be a separable, $\mathfrak{X}$-$C^*$-algebra with property (UBS) and $\psi : A \rightarrow \mathcal{Q}(B)$ be an $\mathfrak{X}$-equivariant, completely positive map. Then there is an $\mathfrak{X}$-equivariant, completely positive map $\phi : A \rightarrow \mathcal{M}(B)$ such that $q_B \circ \phi=\psi$.  
\end{proposition}

Several remarks are in order. Firstly, property (UBS) is another regularity condition and is defined in \cite{gabe:lifting}*{Definition 4.12}. Secondly, let $\mathfrak{X}$ be a locally compact, Hausdorff space $X$ in the above proposition. Then by the discussion before Lemma \ref{lem:x-multiplier}, $A$ is a separable, nuclear, continuous field of $C^*$-algebras, i.e., $A$ is $C_0(X)$-nuclear. By \cite{gabe:lifting}*{Example 4.14}, $B$ automatically has property (UBS) in this case. The lift $\phi$ is $C_0(X)$-linear by Lemma \ref{lem:x-multiplier} and therefore $\psi$ itself is  $C_0(X)$-linear, by Lemma \ref{lem:x-calkin}. Thirdly, the proof of \cite{gabe:lifting}*{Proposition 5.4}, shows that if $\psi \in \mathrm{cp}^1(X;A,\mathcal{Q}(B))$ then $\phi \in \mathrm{cp}^1(X;A,\mathcal{M}(B))$, allowing a bit abuse of notation, of course, because none of the codomain algebras are $C_0(X)$-algebras. Let us summarize our discussion.

\begin{proposition}
Let $A$ and $B$ be two $C_0(X)$-algebras, $A$ be separable and $C_0(X)$-nuclear, $B$ be separable and $\psi \in \mathrm{cp}^1(X;A,\mathcal{Q}(B))$. Then there is a $\phi \in \mathrm{cp}^1(X;A,\mathcal{M}(B))$ such that $q_B \circ \phi=\psi$.
\end{proposition}

We note, however, in this form, we are able to prove the above proposition differently than the proof in \cite{gabe:lifting}. More precisely, although $\mathcal{Q}(B)$ is not a $C_0(X)$-algebra and one cannot therefore apply Proposition \ref{prop:c0x-choi-effros}, one can, however, still run the proof as described above, especially, Lemma \ref{lem:matrix-lifts}. Thus we obtain another justification of the assumption of Lemma \ref{lem:lift-calkin}. Finally, we have the following (special case of the) $\mathfrak{X}$-equivariant Choi-Effros lifting theorem. Before stating it, let us recall that an $\mathfrak{X}$-$C^*$-extension is a short exact sequence 
\[
0 \rightarrow I \rightarrow B \rightarrow D \rightarrow 0,  
\]of $\mathfrak{X}$-$C^*$-algebras and $\mathfrak{X}$-equivariant, $*$-homomorphims such that for each $U \in \mathbb{O}(\mathfrak{X})$, the sequence
\[
0 \rightarrow I(U) \rightarrow B(U) \rightarrow D(U) \rightarrow 0  
\]is exact.

\begin{theorem}[{\cite{gabe:lifting}*{Theorem 5.6}}]\label{theo:x-lifting}
Let $\mathfrak{X}$ be a topological space and $A$ be a separable, nuclear, lower semi-continuous $\mathfrak{X}$-$C^*$-algebra. Let 
\[
0 \rightarrow I \rightarrow B \xrightarrow{p} D \rightarrow 0  
\]be an $\mathfrak{X}$-$C^*$-extension, such that $I$ is separable and has property (UBS). Let $\psi : A \rightarrow D$ be an $\mathfrak{X}$-equivariant, completely positive map. Then there is an $\mathfrak{X}$-equivariant, completely positive map $\phi : A \rightarrow B$ such that $p \circ \phi=\psi$.  
\end{theorem}

Again several remarks are in order. Firstly, let $\mathfrak{X}$ be a locally compact, Hausdorff space $X$ in the above theorem. Then $A$ is a separable, nuclear, continuous field of $C^*$-algebras, i.e., $A$ is $C_0(X)$-nuclear. Secondly, 
\[
0 \rightarrow I \rightarrow B \rightarrow D \rightarrow 0  
\]is a $C_0(X)$-extension with $I$ separable. Finally, by the remarks just before Lemma \ref{lem:x-multiplier}, $\psi$ is a $C_0(X)$-linear, completely positive map. Thus Theorem \ref{theo:x-lifting}, in this very special case, is a slightly weaker version of Proposition \ref{prop:c0x-choi-effros}, in that both $\phi$ and $\psi$ being contractions is omitted, which is, however, easy to deduce from the proof.

To conclude, let us note that \cite{gabe:lifting}*{Proposition 5.5} is similarly related to our Proposition \ref{prop:c0x-busby}; to show equivalence of the latter with the former, one needs to show that the pullback in the former proposition is finitely lower semi-continuous and upper semi-continuous. In this regard, our Proposition \ref{prop:c0x-busby} seems to be slightly more efficient.


\begin{bibdiv}
\begin{biblist}

\bib{akemann-pedersen-tomiyama:multipliers}{article}{
      author={Akemann, C.A.},
      author={Pedersen, G.K.},
      author={Tomiyama, J.},
       title={Multipliers of {$C^*$}-algebras},
        date={1973},
        ISSN={0022-1236},
     journal={J. Functional Analysis},
      volume={13},
       pages={277\ndash 301},
         url={https://doi.org/10.1016/0022-1236(73)90036-0},
      review={\MR{470685}},
}

\bib{arveson:essentially-normal}{article}{
      author={Arveson, W.},
       title={A note on essentially normal operators},
        date={1974},
        ISSN={0035-8975},
     journal={Proc. Roy. Irish Acad. Sect. A},
      volume={74},
       pages={143\ndash 146},
        note={Spectral Theory Symposium (Trinity College, Dublin, 1974)},
      review={\MR{365217}},
}

\bib{arveson:extensions}{article}{
      author={Arveson, W.},
       title={Notes on extensions of {$C^*$}-algebras},
        date={1977},
        ISSN={0012-7094,1547-7398},
     journal={Duke Math. J.},
      volume={44},
      number={2},
       pages={329\ndash 355},
         url={http://projecteuclid.org/euclid.dmj/1077312235},
      review={\MR{438137}},
}

\bib{bauval:nuclear-rkk}{article}{
      author={Bauval, A.},
       title={{$RKK(X)$}-nucl\'{e}arit\'{e} (d'apr\`es {G}. {S}kandalis)},
        date={1998},
        ISSN={0920-3036},
     journal={$K$-Theory},
      volume={13},
      number={1},
       pages={23\ndash 40},
         url={https://doi.org/10.1023/A:1007727426701},
      review={\MR{1610242}},
}

\bib{brown-douglas-fillmore:k-homology}{article}{
      author={Brown, L.G.},
      author={Douglas, R.G.},
      author={Fillmore, P.A.},
       title={Extensions of {$C\sp*$}-algebras and {$K$}-homology},
        date={1977},
        ISSN={0003-486X},
     journal={Ann. of Math. (2)},
      volume={105},
      number={2},
       pages={265\ndash 324},
         url={https://doi.org/10.2307/1970999},
      review={\MR{458196}},
}

\bib{blanchard:deformations}{article}{
      author={Blanchard, E.},
       title={D\'{e}formations de {$C^*$}-alg\`ebres de {H}opf},
        date={1996},
        ISSN={0037-9484},
     journal={Bull. Soc. Math. France},
      volume={124},
      number={1},
       pages={141\ndash 215},
         url={http://www.numdam.org/item?id=BSMF_1996__124_1_141_0},
      review={\MR{1395009}},
}

\bib{blanchard:subtriviality}{article}{
      author={Blanchard, E.},
       title={Subtriviality of continuous fields of nuclear {$C^*$}-algebras},
        date={1997},
        ISSN={0075-4102},
     journal={J. Reine Angew. Math.},
      volume={489},
       pages={133\ndash 149},
         url={https://doi.org/10.1515/crll.1997.489.133},
      review={\MR{1461207}},
}

\bib{borys:furstenberg}{misc}{
      author={Borys, C.},
       title={The {F}urstenberg boundary of a groupoid},
        date={2020},
}

\bib{brown:continuity}{article}{
      author={Brown, L.G.},
       title={Continuity of actions of groups and semigroups on {B}anach
  spaces},
        date={2000},
        ISSN={0024-6107,1469-7750},
     journal={J. London Math. Soc. (2)},
      volume={62},
      number={1},
       pages={107\ndash 116},
         url={https://doi.org/10.1112/S0024610700001058},
      review={\MR{1771854}},
}

\bib{choi-effros:lifting}{article}{
      author={Choi, M.D.},
      author={Effros, E.G.},
       title={The completely positive lifting problem for {$C\sp*$}-algebras},
        date={1976},
        ISSN={0003-486X},
     journal={Ann. of Math. (2)},
      volume={104},
      number={3},
       pages={585\ndash 609},
         url={https://doi.org/10.2307/1970968},
      review={\MR{417795}},
}

\bib{echterhoff-williams:c_0-actions}{article}{
      author={Echterhoff, S.},
      author={Williams, D.P.},
       title={Crossed products by {$C_0(X)$}-actions},
        date={1998},
        ISSN={0022-1236,1096-0783},
     journal={J. Funct. Anal.},
      volume={158},
      number={1},
       pages={113\ndash 151},
         url={https://doi.org/10.1006/jfan.1998.3295},
      review={\MR{1641562}},
}

\bib{forough-gardella:absorption}{misc}{
      author={Forough, M.},
      author={Gardella, E.},
       title={Equivariant bundles and absorption},
        date={2021},
}

\bib{forough-gardella-thomsen:lifting}{misc}{
      author={Forough, M.},
      author={Gardella, E.},
      author={Thomsen, K.},
       title={Asymptotic lifting for completely positive maps},
        date={2021},
}

\bib{fieux:pi-algebra}{article}{
      author={Fieux, E.},
       title={Classes caract\'{e}ristiques d'une {$\Pi$}-alg\`ebre et suite
  spectrale en {$K$}-th\'{e}orie bivariante},
        date={1991},
        ISSN={0920-3036,1573-0514},
     journal={$K$-Theory},
      volume={5},
      number={1},
       pages={71\ndash 96},
         url={https://doi.org/10.1007/BF00538880},
      review={\MR{1141336}},
}

\bib{gabe:lifting}{article}{
      author={Gabe, J.},
       title={Lifting theorems for completely positive maps},
        date={2022},
        ISSN={1661-6952,1661-6960},
     journal={J. Noncommut. Geom.},
      volume={16},
      number={2},
       pages={391\ndash 421},
         url={https://doi.org/10.4171/jncg/479},
      review={\MR{4478259}},
}

\bib{gabe-szabo:stable-uniqueness}{misc}{
      author={Gabe, J.},
      author={Szabó, G.},
       title={The stable uniqueness theorem for equivariant {K}asparov theory},
        date={2022},
}

\bib{gabe-szabo:dynamical-kirchberg-phillips}{article}{
      author={Gabe, J.},
      author={Szabó, G.},
       title={The dynamical {K}irchberg-{P}hillips theorem},
        date={2024},
     journal={Acta Math.},
      volume={232},
      number={1},
       pages={1\ndash 77},
         url={https://dx.doi.org/10.4310/ACTA.2024.v232.n1.a1},
}

\bib{higson-kasparov:e-theory}{article}{
      author={Higson, N.},
      author={Kasparov, G.},
       title={{$E$}-theory and {$KK$}-theory for groups which act properly and
  isometrically on {H}ilbert space},
        date={2001},
        ISSN={0020-9910,1432-1297},
     journal={Invent. Math.},
      volume={144},
      number={1},
       pages={23\ndash 74},
         url={https://doi.org/10.1007/s002220000118},
      review={\MR{1821144}},
}

\bib{hirshberg-rordam-winter:self-absorption}{article}{
      author={Hirshberg, I.},
      author={R{\o}rdam, M.},
      author={Winter, W.},
       title={{$\scr C_0(X)$}-algebras, stability and strongly self-absorbing
  {$C^*$}-algebras},
        date={2007},
        ISSN={0025-5831,1432-1807},
     journal={Math. Ann.},
      volume={339},
      number={3},
       pages={695\ndash 732},
         url={https://doi.org/10.1007/s00208-007-0129-8},
      review={\MR{2336064}},
}

\bib{kasparov:operator-k}{article}{
      author={Kasparov, G.G.},
       title={The operator {$K$}-functor and extensions of {$C\sp{\ast}
  $}-algebras},
        date={1980},
        ISSN={0373-2436},
     journal={Izv. Akad. Nauk SSSR Ser. Mat.},
      volume={44},
      number={3},
       pages={571\ndash 636, 719},
      review={\MR{582160}},
}

\bib{kasparov:equivariant-kk}{article}{
      author={Kasparov, G.G.},
       title={Equivariant {$KK$}-theory and the {N}ovikov conjecture},
        date={1988},
        ISSN={0020-9910},
     journal={Invent. Math.},
      volume={91},
      number={1},
       pages={147\ndash 201},
         url={https://doi.org/10.1007/BF01404917},
      review={\MR{918241}},
}

\bib{kirchberg:classification}{incollection}{
      author={Kirchberg, E.},
       title={Das nicht-kommutative {M}ichael-{A}uswahlprinzip und die
  {K}lassifikation nicht-einfacher {A}lgebren},
        date={2000},
   booktitle={{$C^*$}-algebras ({M}\"{u}nster, 1999)},
   publisher={Springer, Berlin},
       pages={92\ndash 141},
      review={\MR{1796912}},
}

\bib{kasparov-skandalis:group-actions-buildings}{article}{
      author={Kasparov, G.G.},
      author={Skandalis, G.},
       title={Groups acting on buildings, operator {$K$}-theory, and
  {N}ovikov's conjecture},
        date={1991},
        ISSN={0920-3036},
     journal={$K$-Theory},
      volume={4},
      number={4},
       pages={303\ndash 337},
         url={https://doi.org/10.1007/BF00533989},
      review={\MR{1115824}},
}

\bib{kirchberg-wassermann:continuous-bundles}{article}{
      author={Kirchberg, E.},
      author={Wassermann, S.},
       title={Operations on continuous bundles of {$C^*$}-algebras},
        date={1995},
        ISSN={0025-5831,1432-1807},
     journal={Math. Ann.},
      volume={303},
      number={4},
       pages={677\ndash 697},
         url={https://doi.org/10.1007/BF01461011},
      review={\MR{1359955}},
}

\bib{le_gall:equivariant-kk}{article}{
      author={Le~Gall, P-Y.},
       title={Th\'{e}orie de {K}asparov \'{e}quivariante et groupo\"{\i}des.
  {I}},
        date={1999},
        ISSN={0920-3036},
     journal={$K$-Theory},
      volume={16},
      number={4},
       pages={361\ndash 390},
         url={https://doi.org/10.1023/A:1007707525423},
      review={\MR{1686846}},
}

\bib{michael:selection}{article}{
      author={Michael, E.},
       title={Continuous selections. {I}},
        date={1956},
        ISSN={0003-486X},
     journal={Ann. of Math. (2)},
      volume={63},
       pages={361\ndash 382},
         url={https://doi.org/10.2307/1969615},
      review={\MR{77107}},
}

\bib{meyer-nest:bootstrap-class}{article}{
      author={Meyer, R.},
      author={Nest, R.},
       title={{$C^*$}-algebras over topological spaces: the bootstrap class},
        date={2009},
        ISSN={1867-5778,1867-5786},
     journal={M\"{u}nster J. Math.},
      volume={2},
       pages={215\ndash 252},
      review={\MR{2545613}},
}

\bib{manuilov-thomsen:e-theory-special-case}{article}{
      author={Manuilov, V.},
      author={Thomsen, K.},
       title={{$E$}-theory is a special case of {$KK$}-theory},
        date={2004},
        ISSN={0024-6115,1460-244X},
     journal={Proc. London Math. Soc. (3)},
      volume={88},
      number={2},
       pages={455\ndash 478},
         url={https://doi.org/10.1112/S0024611503014436},
      review={\MR{2032515}},
}

\bib{muhly-williams:renault-theorem}{book}{
      author={Muhly, P.S.},
      author={Williams, D.P.},
       title={Renault's equivalence theorem for groupoid crossed products},
      series={New York Journal of Mathematics. NYJM Monographs},
   publisher={State University of New York, University at Albany, Albany, NY},
        date={2008},
      volume={3},
      review={\MR{2547343}},
}

\bib{popescu:equivariant-e-theory}{article}{
      author={Popescu, R.},
       title={Equivariant {$E$}-theory for groupoids acting on
  {$C^*$}-algebras},
        date={2004},
        ISSN={0022-1236},
     journal={J. Funct. Anal.},
      volume={209},
      number={2},
       pages={247\ndash 292},
         url={https://doi.org/10.1016/j.jfa.2003.04.001},
      review={\MR{2044224}},
}

\bib{proietti-yamashita:homology-1}{article}{
      author={Proietti, V.},
      author={Yamashita, M.},
       title={Homology and {$K$}-theory of dynamical systems {I}.
  {T}orsion-free ample groupoids},
        date={2022},
        ISSN={0143-3857,1469-4417},
     journal={Ergodic Theory Dynam. Systems},
      volume={42},
      number={8},
       pages={2630\ndash 2660},
         url={https://doi.org/10.1017/etds.2021.50},
      review={\MR{4448401}},
}

\bib{renault:book-groupoid-approach}{book}{
      author={Renault, J.},
       title={A groupoid approach to {$C^*$}-algebras},
      series={Lecture Notes in Mathematics},
   publisher={Springer, Berlin},
        date={1980},
      volume={793},
        ISBN={3-540-09977-8},
      review={\MR{584266}},
}

\bib{renault:crossed-products}{article}{
      author={Renault, J.},
       title={Repr\'{e}sentation des produits crois\'{e}s d'alg\`ebres de
  groupo\"{\i}des},
        date={1987},
        ISSN={0379-4024},
     journal={J. Operator Theory},
      volume={18},
      number={1},
       pages={67\ndash 97},
      review={\MR{912813}},
}

\bib{rieffel:continuous-fields}{article}{
      author={Rieffel, M.A.},
       title={Continuous fields of {$C^*$}-algebras coming from group cocycles
  and actions},
        date={1989},
        ISSN={0025-5831},
     journal={Math. Ann.},
      volume={283},
      number={4},
       pages={631\ndash 643},
         url={https://doi.org/10.1007/BF01442857},
      review={\MR{990592}},
}

\bib{raeburn-williams:pullbacks}{article}{
      author={Raeburn, I.},
      author={Williams, D.P.},
       title={Pull-backs of {$C^*$}-algebras and crossed products by certain
  diagonal actions},
        date={1985},
        ISSN={0002-9947,1088-6850},
     journal={Trans. Amer. Math. Soc.},
      volume={287},
      number={2},
       pages={755\ndash 777},
         url={https://doi.org/10.2307/1999675},
      review={\MR{768739}},
}

\bib{suzuki:equivariant-absorption}{article}{
      author={Suzuki, Y.},
       title={Equivariant {$\mathcal{O}_2$}-absorption theorem for exact
  groups},
        date={2021},
        ISSN={0010-437X,1570-5846},
     journal={Compos. Math.},
      volume={157},
      number={7},
       pages={1492\ndash 1506},
         url={https://doi.org/10.1112/s0010437x21007168},
      review={\MR{4275465}},
}

\bib{szabo:self-absorbing-3}{article}{
      author={Szab\'{o}, G.},
       title={Strongly self-absorbing {${\rm C}^\ast$}-dynamical systems,
  {III}},
        date={2017},
        ISSN={0001-8708,1090-2082},
     journal={Adv. Math.},
      volume={316},
       pages={356\ndash 380},
         url={https://doi.org/10.1016/j.aim.2017.06.008},
      review={\MR{3672909}},
}

\bib{szabo:equivariant-kirchberg-phillips}{article}{
      author={Szab\'{o}, G.},
       title={Equivariant {K}irchberg-{P}hillips-type absorption for amenable
  group actions},
        date={2018},
        ISSN={0010-3616,1432-0916},
     journal={Comm. Math. Phys.},
      volume={361},
      number={3},
       pages={1115\ndash 1154},
         url={https://doi.org/10.1007/s00220-018-3110-3},
      review={\MR{3830263}},
}

\bib{szabo:self-absorbing-1}{article}{
      author={Szab\'{o}, G.},
       title={Strongly self-absorbing {$\rm C^*$}-dynamical systems},
        date={2018},
        ISSN={0002-9947,1088-6850},
     journal={Trans. Amer. Math. Soc.},
      volume={370},
      number={1},
       pages={99\ndash 130},
         url={https://doi.org/10.1090/tran/6931},
      review={\MR{3717976}},
}

\bib{szabo:self-absorbing-2}{article}{
      author={Szab\'{o}, G.},
       title={Strongly self-absorbing {$\rm C^*$}-dynamical systems. {II}},
        date={2018},
        ISSN={1661-6952,1661-6960},
     journal={J. Noncommut. Geom.},
      volume={12},
      number={1},
       pages={369\ndash 406},
         url={https://doi.org/10.4171/JNCG/279},
      review={\MR{3782062}},
}

\bib{thomsen:equivariant-kk}{article}{
      author={Thomsen, K.},
       title={Equivariant {$KK$}-theory and {$C^*$}-extensions},
        date={2000},
        ISSN={0920-3036,1573-0514},
     journal={$K$-Theory},
      volume={19},
      number={3},
       pages={219\ndash 249},
         url={https://doi.org/10.1023/A:1007853018475},
      review={\MR{1756259}},
}

\bib{tu:baum-connes-amenable}{article}{
      author={Tu, J.L.},
       title={La conjecture de {B}aum-{C}onnes pour les feuilletages
  moyennables},
        date={1999},
        ISSN={0920-3036,1573-0514},
     journal={$K$-Theory},
      volume={17},
      number={3},
       pages={215\ndash 264},
         url={https://doi.org/10.1023/A:1007744304422},
      review={\MR{1703305}},
}

\bib{tikuisis-white-winter:quasidiagonality}{article}{
      author={Tikuisis, A.},
      author={White, S.},
      author={Winter, W.},
       title={Quasidiagonality of nuclear {$C^\ast$}-algebras},
        date={2017},
        ISSN={0003-486X},
     journal={Ann. of Math. (2)},
      volume={185},
      number={1},
       pages={229\ndash 284},
         url={https://doi.org/10.4007/annals.2017.185.1.4},
      review={\MR{3583354}},
}

\bib{williams:book-crossed-products}{book}{
      author={Williams, D.P.},
       title={Crossed products of {$C{^\ast}$}-algebras},
      series={Mathematical Surveys and Monographs},
   publisher={American Mathematical Society, Providence, RI},
        date={2007},
      volume={134},
        ISBN={978-0-8218-4242-3; 0-8218-4242-0},
         url={https://doi.org/10.1090/surv/134},
      review={\MR{2288954}},
}

\end{biblist}
\end{bibdiv}
\end{document}